\documentclass[a4paper,12pt,reqno,oneside]{amsart} 

\usepackage{setspace} 
\usepackage{typearea} 
\usepackage{amscd,amssymb,verbatim} 

\usepackage{xr-hyper}
\usepackage{cite}
\usepackage[colorlinks,backref,final,bookmarksnumbered,bookmarks]{hyperref}
\usepackage[backrefs]{amsrefs}

\usepackage{ifthen}
\newcommand{\arxiv}[2][]{\ifthenelse{\equal{#1}{}}
{\href{http://arxiv.org/abs/#2}{\tt arXiv:#2}}
{\href{http://arxiv.org/abs/math/#2}{\tt arXiv:math.#1/#2}}}

\theoremstyle{plain}
\newtheorem{maintheorem}{Theorem}
\newtheorem{theorem}[subsection]{Theorem}
\newtheorem*{theorem*}{Theorem}
\newtheorem{lemma}[subsection]{Lemma}

\newtheorem*{corollary*}{Corollary}
\newtheorem*{example*}{Example}
\newtheorem{proposition}[subsection]{Proposition}

\newtheorem{addendum}[subsection]{Addendum}
\newtheorem{conjecture}[subsection]{Conjecture}

\theoremstyle{definition}
\newtheorem{example}[subsection]{Example}

\newtheoremstyle{remark}
{}{}{}{}{\itshape}{}{ }{\thmname{#1}\thmnumber{ \itshape #2.}}
\theoremstyle{remark}
\newtheorem{remark}[subsection]{Remark}

\DeclareMathOperator{\id}{id}

\DeclareMathOperator{\Fr}{Fr}
\DeclareMathOperator{\Imm}{Imm}
\DeclareMathOperator{\Emb}{Emb}

\DeclareMathOperator{\sfr}{sf}

\DeclareMathOperator{\rel}{rel}

\DeclareMathOperator{\Cone}{Cone}

\def\x{\times}
\def\but{\setminus}
\def\emb{\hookrightarrow}
\def\imm{\looparrowright}
\def\eps{\varepsilon}
\def\phi{\varphi}
\def\tl{\tilde}
\def\emptyset{\varnothing}
\renewcommand{\:}{\colon}
\def\xr#1{\xrightarrow{#1}}

\def\R{\mathbb{R}}
\def\Z{\mathbb{Z}}

\def\C{\mathbb{C}}
\def\E{\mathcal{E}}
\def\Cl#1{\overline{#1}}

\def\RP{\R P} \def\CP{\C P}
\def\KP{\mathbb{K} P} 
\let\widebar\overline \def\Cup{\!\smallsmile\!}

\begin{document}

\title{Projected and near-projected embeddings}
\author{Peter M. Akhmetiev and Sergey A. Melikhov}
\address{HSE Tikhonov Institute of Electronics and Mathematics, 34 Tallinskaya Str., Moscow, 123458, Russia}
\address{IZMIRAN, Moscow region, Troitsk, 142190 Russia}
\email{pmakhmet@mail.ru}
\address{Steklov Mathematical Institute of Russian Academy of Sciences,
ul.\ Gubkina 8, Moscow, 119991 Russia}
\email{melikhov@mi-ras.ru}

\begin{abstract} A stable smooth map $f\:N\to M$ is called {\it $k$-realizable} if its composition
with the inclusion $M\subset M\x\R^k$ is $C^0$-approximable by smooth embeddings; and a {\it $k$-prem} 
if the same composition is $C^\infty$-approximable by smooth embeddings, or equivalently if $f$ lifts 
vertically to a smooth embedding $N\emb M\x\R^k$.

It is obvious that if $f$ is a $k$-prem, then it is $k$-realizable.
We refute the so-called ``Prem Conjecture'' that the converse holds. 
Namely, for each $n=4k+3\ge 15$ there exists a stable smooth immersion $S^n\imm\R^{2n-7}$
that is $3$-realizable but is not a $3$-prem.

We also prove the converse in a wide range of cases.
A $k$-realizable stable smooth fold map $N^n\to M^{2n-q}$ is a $k$-prem if $q\le n$ and $q\le 2k-3$; 
or if $q<n/2$ and $k=1$; or if $q\in\{2k-1,\,2k-2\}$ and $k\in\{2,4,8\}$ and $n$ is sufficiently large.
\end{abstract}

\maketitle

\section{Introduction}

\subsection{$k$-Prems} \label{prem}
We call a map $f\:N\to M$ a (PL/smooth) {\it $k$-prem} ($k$-codimensionally {\bf pr}ojected {\bf em}bedding) if 
it factors into the composition of some (PL/smooth) embedding $N\emb M\x\R^k$ and the projection $M\x\R^k\to M$.
For example, a constant map $f$ is a $k$-prem if and only if $N$ embeds in $\R^k$. 
The abbreviation ``prem'' was coined by A. Sz\H ucz (see \cite{Akh}, \cite{Sz}).
Let us mention some results on the subject (further references can be found in \cite{A2}, \cite{M4}, 
\cite{M5}).

\begin{enumerate}
\item It is well-known and easy to see that stable%
\footnote{See \cite{GG} concerning stable (i.e.\ $C^\infty$-stable) smooth maps and
\cite{M5}*{Part I, Appendix B} concerning stable PL maps.}
smooth functions $S^1\to\R^1$ are smooth $1$-prems (cf.\ Example \ref{1-dim} below).
On the other hand, there exist stable PL functions on trees, ``letter H''$\to\R^1$ and ``letter X''$\to\R^1$ 
that are not PL $1$-prems (Siek\l ucki \cite{Si}; see also another proof in \cite{ARS}*{\S3}).

\item \label{Kulinich} It is not hard to show that stable smooth functions $M^2\to\R^1$ on a orientable 
surface are smooth $2$-prems (\cite{Kul}; see also \cite{M2}*{Proof of the Yamamoto--Akhmetiev Theorem}).
Tarasov and Vyalyi constructed a stable PL function $f\:M^2\to\R^1$ on an orientable surface of a high genus
that is not a PL $2$-prem \cite{TV}.
They also proved that stable PL functions $S^2\to\R^1$ are $2$-prems.

\item Stable smooth maps $M^2\to\R^2$ of an orientable surface are smooth $2$-prems (Yamamoto \cite{Ya}).
Stable smooth maps $S^2\to S^2$ are smooth $2$-prems (Yamamoto--Akhmetiev; see \cite{M2}).
Stable PL/smooth maps $S^n\to S^n$, $n\ge 3$, are PL/smooth $n$-prems (Melikhov \cite{M5}).

\item Every stable smooth approximation of the composition $S^2\to\RP^2\imm\R^3$ of the $2$-covering and 
the Boy immersion is not a smooth $1$-prem (Akhmetiev).
This follows from the results of \cite{A1} for $C^\infty$-approximations, but the proof there works for 
$C^0$-approximations as well.
See also \cite{Gi}, \cite{CS}, \cite{S}, \cite{Po1}, \cite{Po2}.

\item Every finite covering between spaces that are homotopy equivalent to $n$-polyhedra is 
an $(n+1)$-prem (Hansen \cite{Han}).
Every non-degenerate%
\footnote{A PL map is called {\it non-degenerate} if has no point-inverses of dimension $>0$.}
PL map between compact $n$-polyhedra is a PL $(n+1)$-prem (Melikhov \cite{M5}).
It is easy to see (cf.\ Proposition \ref{1.2} and its proof) that every stable PL/smooth map $N^n\to M^m$ is 
a PL/smooth $(2n+1-m)$-prem.

\item A stable PL map is a $k$-prem if and only if it is a PL $k$-prem (Melikhov \cite{M5}).
\end{enumerate}

For a stable smooth map $f\:N^n\to M^m$, where $N$ is compact and $m\ge n$, it is not hard to show that $f$ 
is a smooth $k$-prem if and only if the composition $N\xr{f}M\subset M\x\R^k$ is $C^\infty$-approximable 
by embeddings (see Proposition \ref{1.2} below).

\subsection{$k$-Realizable maps}
We say that a map $f\:N^n\to M^m$ is (PL/smoothly) {\it $k$-realizable} (realization by
embeddings) if the composition $N\xr{f}M\subset M\x\R^k$ is $C^0$-approximable
by (PL/smooth) embeddings.
Topological, PL and smooth realizability are equivalent in the metastable range 
$2(m+k)\ge 3(n+1)$ \cite{Hae}.

Since every compact subset of $\R^k$ can be brought into an arbitrarily small neighborhood 
of the origin by a diffeomorphism of $\R^k$, every $k$-prem is $k$-realizable.

It quite easy to construct non-stable maps that are $k$-realizable but not $k$-prems
(see Examples \ref{1-dim}, \ref{2-dim}, \ref{3-dim} below).

The following conjecture appears in the first paragraph of the 2002 paper \cite{ARS}:

\begin{conjecture}[The Prem Conjecture] Smoothly $k$-realizable general position%
\footnote{The exact meaning of ``general position maps'' is not discussed in \cite{ARS}, but this term 
must have been imported, via \cite{Akh} and \cite{Sz2}, from \cite{Sz}, where ``generic maps'' are
defined as Boardman maps with normal crossings (in the terminology of \cite{GG}).
Such maps form an open dense subset of $C^\infty(N,M)$ and hence include all stable maps (see \cite{GG}).}
smooth maps $N^n\to\R^m$ are smooth $k$-prems, at least in the metastable range $2(m+k)\ge 3(n+1)$.
\end{conjecture}

It was proved in \cite{ARS} that a $1$-realizable stable smooth map $N^n\to\R^{2n-1}$, $n\ge 3$, 
is a smooth $1$-prem.
It was also proved in \cite{ARS} that a smoothly $1$-realizable stable smooth map of an orientable surface 
$M^2\to\R^3$ is a smooth $1$-prem (the idea of proof for $M^2=S^2$ appears already in \cite{A1}).

\begin{theorem} \label{1.1}
There exists a stable PL function $f\:M\to\R^1$ on some closed orientable $2$-manifold such that $f$ is 
PL $2$-realizable but not a PL $2$-prem. 
\end{theorem}

This is an easy consequence of the above mentioned result of Tarasov and Vyalyi \cite{TV}.

\begin{proof}
Let $f\:M^2\to\R^1$ be the stable PL function of \cite{TV} that is not a PL $2$-prem.
Since $f$ is approximable by stable smooth functions, which are $2$-prems (see item (\ref{Kulinich}) in 
\S\ref{prem}), $f$ is $2$-realizable.
\end{proof}

\subsection{Main results} \label{1.A}
We construct a counterexample to the Prem Conjecture:

\begin{maintheorem}\label{th1} For each $n=4k+3\ge 15$ there exists a stable smooth
immersion $S^n\imm\R^{2n-7}$ that is $3$-realizable, but is not a $3$-prem.
\end{maintheorem}

We also show that the Prem Conjecture holds in a wide range of dimensions:

\begin{maintheorem}\label{th2} Let $N^n$ be a compact smooth manifold (resp.\ a compact polyhedron),
$M^m$ a smooth (PL) manifold and $f\:N\to M$ a stable smooth (PL) map, where $m\ge n$ and
$2n-m\le 2k-3$.
In the smooth case, assume additionally that either $f$ is a fold map or $3n-2m\le k$.
Then $f$ is $k$-realizable if and only if it is a smooth (PL) $k$-prem.
\end{maintheorem}

Theorem \ref{th2} implies the first author's 2004 conjecture \cite{A2}*{Conjecture 1.9}: a generic smooth map 
$S^n\to S^n$ is $k$-realizable if and only if it is an $k$-prem, as long as $n\le 2k-3$ (except
that the definition of ``generic'' in \cite{A2} has to be improved).
On the other hand, Theorem \ref{th2} for stable smooth maps falls short of the second author's 2004 
announcement in \cite{M2}, where it was expected to hold under the weaker restriction $4n-3m\le k$.

The proof of these results is based on the following two criteria.
Let $f$ be a stable PL or smooth map $f\:N\to M$.
\begin{itemize}
\item By recent work of the second author \cite{M5}, $f$ is a $k$-prem if and only if 
$\Delta_f:=\{(x,y)\in N\x N\but\Delta_N\mid f(x)=f(y)\}$ admits a $\Z/2$-equivariant map 
to $S^{k-1}$, assuming that either $f$ is a fold map or $3n-2m\le k$ in the smooth case.
\item Building on some previous work \cite{RS}, \cite{A2}, \cite{M1}, we show that $f$ is $k$-realizable 
if and only if $\Delta_f$ admits a {\it stable} $\Z/2$-map to $S^{k-1}$, i.e.\ for some $x$ there exists 
a $\Z/2$-map $\Delta_f*S^x\to S^{k-1}*S^x=S^{k+x}$.
\end{itemize}
Due to these two criteria, in the stable range we prove the equivalence of the two notions
($k$-realizable and $k$-prem), whereas in the unstable range we have a chance to realize 
the difference known from homotopy theory by a geometric example.
The latter task is nontrivial, which is why there is the huge gap in dimensions between 
Theorems \ref{th1} and \ref{th2}.

The remainder of the paper is largely devoted to the study of this gap.

It turns out pretty quickly that we could not have had $k=1$ in Theorem \ref{th1}:

\begin{maintheorem}\label{th3} Let $N^n$ be a compact smooth manifold (resp.\ a compact polyhedron), 
$M^m$ a smooth (PL) manifold and suppose that $2m>3n$.
A stable smooth (PL) map $f\:N\to M$ is $1$-realizable if and only if it is a smooth (PL) $1$-prem.
\end{maintheorem}

In fact, ``stable equivariant maps'' as defined above are not the same as ``stable equivariant maps'' in 
the usual sense of homotopy theory.
For example, our actions of $\Z/2$ do not have any fixed points, but they must always be present in 
the traditional setup of equivariant homotopy theory.
In \S\ref{s3} we reformulate our problem in the traditional algebraic setup, which already brings some progress.

In a relative situation, we are able to produce a simpler example, with $k=2$ instead of $k=3$: 

\begin{maintheorem}\label{th4} For each $n\ge 7$ there exists a stable smooth immersion
$f\:S^n\imm\R^{2n-2}$ and its $2$-realization $g\:S^n\emb\R^{2n}$,
not isotopic through $2$-realizations of $f$ to any vertical lift of $f$.
\end{maintheorem}

Also, due to a little gap between the ``stable range'' in the traditional algebraic setup
and the ``stable range'' in our initial approach above, with some numeric luck regarding
the value of $k$ we are able to go one dimension deeper:

\begin{maintheorem}\label{th5} Let $N^n$ be a compact smooth manifold (resp.\ a compact polyhedron), 
$M^{2n-2k+2}$ a smooth (PL) manifold and $f\:N\to M$ a stable smooth (PL) map, where $n\ge 2k-1$.
In the smooth case, assume additionally that either $f$ is a fold map or $n\ge 3k-4$.
If $f\:N\to M$ is $k$-realizable and $k\in\{2,4,8\}$, then $f$ is a smooth (PL) $k$-prem.
\end{maintheorem}

In \S\ref{s4} we further reformulate a part of the problem, which has by now became fully
algebraic, in geometric terms (using the Pontryagin--Thom construction).
This enables us to construct a secondary obstruction in stable cohomotopy and prove 
its vanishing (Theorem \ref{4.7}), which yields that the stable and the unstable cohomotopy 
Euler classes of the sum of $k$ copies of a line bundle over an $n$-polyhedron $X$
are equally strong in the first unstable dimension $n=2k-1$ when $k$ is even.

As a consequence, we can sometimes go down by one more dimension:

\begin{maintheorem}\label{th6} Let $N^n$ be a compact smooth manifold (resp.\ a compact polyhedron), 
$M^{2n-2k+1}$ a smooth (PL) manifold and $f\:N\to M$ a stable smooth (PL) map, where $n\ge 2k+1$.
In the smooth case, assume additionally that either $f$ is a fold map or $n\ge 3k-2$.
If $f\:N\to M$ is $k$-realizable and $k\in\{2,4,8\}$, then $f$ is a smooth (PL) $k$-prem.
\end{maintheorem}

This theorem disproves the first author's 2004 conjecture \cite{A2}*{Conjecture 1.10}: ``There exists 
a generic immersion $f\:S^7\to\R^{11}$ such that the composition of $f$ with the inclusion 
$\R^{11}\subset\R^{13}$ is discretely realizable but not $C^1$-approximable by embedding.''

\subsection{Examples and remarks} \label{1.B}
The following simple examples explain why one restricts attention to generic/stable maps in the Prem Conjecture.

\begin{example}[A $1$-realizable Morse function on $S^1$ that is not a $1$-prem]\label{1-dim}
Let us consider the composition $S^1\xr{p}S^1\xr{\phi}\R$ of the $2$-covering
$p$ and the simplest Morse function $\phi$ with two critical points $x$, $y$.
If $q\:S^1\to\R$ is a map such that $q\x(\phi p)\:S^1\to\R\x\R$ is an
embedding, then $q|_{p^{-1}(x)}\:S^0\emb\R$ is isotopic to
$q|_{p^{-1}(y)}\:S^0\emb\R$ for each choice of the homeomorphism between
$p^{-1}(x)$ and $p^{-1}(y)$.
Hence $S^0$ can be everted by an isotopy within $\R^1$, which is absurd.

On the other hand, $\phi p$ is $1$-realizable, since every stable smooth function
$f\:S^1\to\R$ is a $1$-prem.
Indeed, let $y\in\R$ be the absolute maximum of $f$.
If $\eps>0$ is sufficiently small, $f^{-1}([-\infty,y-\eps))$ is an arc $I$.
The graph $\Gamma(f|_I)\:I\emb I\x [-\infty,y-\eps]$ combines with the graph
$\Gamma(f|_{\Cl{S^1\but I}})\:\Cl{S^1\but I}\emb I\x [y-\eps,+\infty)$ into
an embedding $S^1\to I\x\R$ that projects onto $f$.
\end{example}

\begin{example}[A $2$-realizable Morse function on $S^2$ that is not a $2$-prem]\label{2-dim}
The complete graph on $5$ vertices $K_5$ contains a Hamiltonian
(i.e.\ passing once through each vertex) cycle $Z_1$ of length $5$, and the
remaining $5$ edges of $K_5$ form another Hamiltonian cycle $Z_2$.
By switching the thus obtained smoothing of each vertex to the opposite one
(without changing the orientations of the edges), we get an Eulerian
(i.e.\ passing once through each edge) cycle $Z$.
Attaching three $2$-disks to $K_5$ along $Z_1$, $Z_2$ and $Z$, we obtain
a genus two closed surface $M$ and a non-stable Morse function
$f\:M\to\R^1$ whose only degenerate critical levels are: a two-point set $f^{-1}(1)$
(both points are maxima); $K_5=f^{-1}(0)$ (with a saddle point at each vertex of $K_5$); 
a singleton $f^{-1}(-1)$ (which is a minimum).
Since every stable Morse function on an orientable surface is a $2$-prem
(see item \ref{Kulinich} in \S\ref{prem}), $f$ is $2$-realizable.
However $f$ is not a $2$-prem since $K_5$ does not embed in $\R^2$.
\end{example}

\begin{example}[A $3$-realizable branched covering $S^3\to S^3$ that is not a $3$-prem]\label{3-dim}
The join of two copies of the $2$-covering $S^1\to S^1$ is a non-stable map
$S^3\to S^3$.
It factors into the composition of the $2$-covering $p\:S^3\to\RP^3$ and
the $2$-fold covering $q\:\RP^3\to S^3$ branched along the Hopf link.
Since $p$ is not a $3$-prem by the Borsuk--Ulam theorem (see \cite{M2}),
nor is $qp$.
However by \cite{M2}, $qp$ is $3$-realizable.
\end{example}

\begin{proposition}\label{1.2} Let $f\:N^n\to M^m$ be a stable smooth map between smooth manifolds, 
where $m\ge n$ and $N$ is compact.
Then $f$ is a smooth $k$-prem if and only if the composition of $f$ with the inclusion
$j\:M\x\{0\}\emb M\x\R^k$ is $C^\infty$-approximable by embeddings.
\end{proposition}

\begin{proof} Given a smooth map $g\:N\to\R^k$ such that $f\x g\:N\to M\x\R^k$ is
a smooth embedding, since $N$ is compact, for each $\eps>0$ one can find a $\delta>0$ 
such that $f\x\delta g$ is $C^\infty$-$\eps$-close to $jf$.

Conversely, if $f'\x g\:N\to M\x\R^k$ is a smooth embedding, $C^\infty$-$\eps$-close 
to $jf$, then $f'$ is $C^\infty$-$\eps$-close to $f$.
If $\eps>0$ is sufficiently small, since $f$ is stable, there exist diffeomorphisms
$h\:N\to N$ and $H\:M\to M$ such that $fh=Hf'$.
Then $H\x\id_{\R^k}$ takes $f'\x g$ onto $(fh)\x g$, which is therefore a smooth embedding.
Hence $f\x(gh^{-1})$ is also a smooth embedding.
So $f$ is a $k$-prem. 
\end{proof}

\begin{remark} In this remark we use ``generic'' in the sense of \cite{M5}.
Let $N^n$ and $M^m$ be smooth manifolds, where $m\ge n$ and $N$ is compact.

(a) By Mather's theorem \cite{Ma} generic smooth maps $f\:N\to M$ are stable if either $6m\ge 7n-7$, 
or $6m\ge 7n-8$ and $m\le n+3$.
By multijet transversality, generic smooth immersions $N\to M$ are stable \cite{GG}*{III.3.3, III.3.11}.
More generally, by Morin's canonical form and well-known results of Mather and Boardman,
generic corank one maps $N\to M$ are stable (see \cite{M5}*{Part I, Theorem A.2}).
In particular, generic fold maps $N\to M$ are stable.

(b) The ``only if'' part of Proposition \ref{1.2} holds without assuming that $f$ is stable.
The ``if'' part has the following version with a weaker hypothesis and weaker conclusion.
If $f\:N\to M$ is a generic smooth map such that $jf$ is $C^\infty$-approximable by embeddings, then 
$f$ is a topological $k$-prem.
This is proved similarly to Proposition \ref{1.2}, using that $f$ is $C^0$-stable \cite{GWPL}.
\end{remark}

\subsection*{Acknowledgements}
We are grateful to E. A. Kudryavtseva, S.~ Maksymenko, Yu.~ Rudyak, R.~ Sadykov, L.~ Siebenmann and 
S. Tarasov for stimulating conversations and useful remarks.

A preliminary version of this paper was privately circulated as a preprint \cite{AMR}.

\section{The Prem Conjecture} \label{s2}

Let $S^k_\circ$ denote the $k$-sphere endowed with the antipodal involution.
Let $\tl X=X\x X\but\Delta$ with the factor exchanging involution $t$.
Given a map $f\:X\to M$, set $\Delta_f=\{(x,y)\in\tl X\mid f(x)=f(y)\}$.
Then $\tl f\:\tl X\but\Delta_f\to S^{m-1}_\circ$ is defined by
$(x,y)\mapsto\frac{f(x)-f(y)}{||f(x)-f(y)||}$.

\begin{theorem}\label{2.1} {\rm (Melikhov \cite{M5})}
Let $N^n$ be a compact smooth manifold (resp.\ a compact polyhedron), $M^m$ a smooth (PL) manifold, 
and $f\:N\to M$ a stable smooth (PL) map, where $m\ge n$ and $2(m+k)\ge 3(n+1)$.
In the smooth case, assume additionally that either $f$ is a fold map or $3n-2m\le k$.
Then $f$ is a smooth (PL) $k$-prem if and only if there exists an equivariant map
$\Delta_f\xr{\phi} S^{k-1}_\circ$.

Moreover, for any such $\phi$ there exists a smooth (PL) embedding
$g\:N\emb M\x\R^k$ projecting to $f$ and such that $\tl g|_{\Delta_f}$ is
equivariantly homotopic to $\phi$.
\end{theorem}

\begin{remark} \label{2.1'}
In the case where $f$ is a triple point free immersion, Theorem \ref{2.1} is obvious.
Indeed, in this case $\Delta_f\subset N\x N$ projects homeomorphically onto $S_f\subset N$, so any extension of 
the equivariant map $S_f\to S^{k-1}$ to a smooth (PL) map $\psi\:N\to\R^k$ yields a smooth (PL) embedding
$f\x\psi\:N\emb M\x\R^k$.
\end{remark}

\begin{remark} As discussed in \cite{M5}, the hypothesis of stability can be somewhat relaxed in the PL case
and in the non-fold smooth case of Theorem \ref{2.1}. 
Consequently this hypothesis can also be similarly relaxed in all our main results of the positive type.
\end{remark}

\begin{theorem}\label{2.2} Let $X^n$ be a compact polyhedron, $Q^q$ a PL manifold
and $f\:X\to Q$ a PL map, where $2q\ge 3(n+1)$.

(a) {\rm (A. Skopenkov \cite{RS})} When $Q=\R^q$, $f$ is $C^0$-approximable by embeddings 
if and only if $\tl f\:\tl X\but\Delta_f\to S^{q-1}$ extends, after 
an equivariant homotopy, to an equivariant map $\tl X\xr{\phi}S^{q-1}_\circ$.

Moreover, for any such $\phi$ and any $\eps>0$ there exists a PL embedding $g$, 
$C^0$-$\eps$-close to $f$ and such that $\tl g$ and $\phi$ are equivariantly 
homotopic with support in an equivariant regular neighborhood of $\Delta_f$.

(b) {\rm (Melikhov \cite{M1}*{1.7(a,a+)})} 
In general, $f$ is $C^0$-approximable by embeddings 
if and only if $f\x f\:X\x X\to Q\x Q$ is $C^0$-approximable by isovariant maps.

Moreover, there exists a $\delta>0$ such that if $\Phi\:X\x X\to Q\x Q$ is 
an isovariant map, $C^0$-$\delta$-close to $f\x f$, then for each $\eps>0$,
there exists a PL embedding $g$, $C^0$-$\eps$-close to $f$ and such that 
$g\x g$ and $\Phi$ are isovariantly homotopic.
\end{theorem}

\subsection{Proof of Theorem \ref{th2}}
Using Theorem \ref{2.2} we will now prove the following.

\begin{theorem}\label{2.3} Let $X^n$ be a compact polyhedron, $M^m$ a compact PL 
manifold and $f\:X\to M$ a stable PL map, where $2(m+k)\ge 3(n+1)$.
Then $f$ is $k$-realizable if and only if $S^{m-1}_\circ*\Delta_f$ admits
an equivariant map to $S^{m+k-1}_\circ$.
\end{theorem}

\begin{proof} Let $\check X$ denote $\Cl{X\x X\but N}$, where $N$ is
the second derived neighborhood of the diagonal in some equivariant
triangulation of $X\x X$ in which $f\x f\:X\x X\to M\x M$ is simplicial.
Then $(\tl X,\Delta_f)$ equivariantly deformation retracts onto the pair 
$(\check X,\check\Delta_f)$ of compact polyhedra, where 
$\check\Delta_f=\Cl{\Delta_f\but N}$.
Let $R$ be a $\Z/2$-invariant regular neighborhood of $\check\Delta_f$ in $\check X$, so 
that $R\cup N$ is a $\Z/2$-invariant regular neighborhood of $\Delta_f\cup\Delta_X$ in $X$.
Since $f$ is stable, $f\x f$ restricted to $\check X$ is PL transverse to $\Delta_M$.

Let us focus on the case where $M$ is smoothable.
Let $\tau$ be the equivariant normal PL disc bundle of $\Delta_M$ in $M\x M$, whose
total space $D(\tau)$ is a $\Z/2$-invariant regular neighborhood of $\Delta_M$.
If $\phi$ denotes $f\x f|_{\check\Delta_f}\:\check\Delta_f\to\Delta_M$, then $\phi^*\tau$
is the equivariant normal PL disc bundle of $\check\Delta_f$ in $\check X$, with 
total space $D(\phi^*\tau_M)=R$ \cite{RS1}, \cite{BRS}*{\S II.4}.

In the case $M=\R^m$, $\tau$ is equivariantly trivial, so $\phi^*\tau$ is equivariantly 
isomorphic to $\check\Delta_f\x D^m_\circ\to\check\Delta_f$ with the diagonal action 
of $\Z/2$, where $D^m_\circ$ denotes the $m$-ball with the antipodal action of $\Z/2$
(with one fixed point).
With respect to this trivialization, $\tilde f$ restricted to 
$\check\Delta_f\x\partial D^m_\circ$ is precisely the projection onto 
$\partial D^m_\circ=S^{m-1}_\circ$.
Then the restriction $\check f\:\check X\but\check\Delta_f\to S^{m-1}\subset S^{m+k-1}$
of $\tilde f$ extends, after an equivariant homotopy, to an equivariant map 
$\check X\to S^{m+k-1}_\circ$ if and only if the projection 
$\check\Delta_f\x\partial D^m_\circ\to\partial D^m_\circ=S^{m-1}_\circ\subset S^{m+k-1}_\circ$
extends to an equivariant map $\check\Delta_f\x D^m_\circ\to S^{m+k-1}_\circ$.
Or in other words if and only if the inclusion $S^{m-1}_\circ\emb S^{m+k-1}_\circ$ extends
to an equivariant map $\check\Delta_f*S^{m-1}_\circ\to S^{m+k-1}_\circ$.
But since $k>0$, the latter is equivalent to the existence of an equivariant map
$\check\Delta_f*S^{m-1}_\circ\to S^{m+k-1}_\circ$.
Taking into account Theorem \ref{2.2}(a), this completes the proof of the case $M=\R^m$.

Now let us return to the case where $M$ is smoothable.
Let $j\:M\emb M\x\R^k$ be the inclusion, and let $\eps^k$ be the trivial equivariant
PL $k$-disc bundle over $\Delta_M$ (with fibers $D^k_\circ$). 
It is easy to see that $(jf)\x (jf)$ is $C^0$-approximable by isovariant maps if and only if
the canonical map $d_{\phi,\tau}\:D(\phi^*\tau)\to D(\tau)\subset D(\tau\oplus\eps^k)$ is equivariantly homotopic, keeping the total space $S(\phi^*\tau)$ of the boundary sphere 
bundle fixed, to a map $D(\phi^*\tau)\to S(\tau\oplus\eps^k)$.
Since the bundle projection $S(\tau\oplus\eps^k)\to\Delta_M$ is an equivariant fibration,
the latter holds if and only if the canonical map 
$s_{\phi,\tau}\:S(\phi^*\tau)\to S(\tau)\subset S(\tau\oplus\eps^k)$ extends to
an equivariant map $D(\phi^*\tau)\to S(\tau\oplus\eps^k)$ lying over the map 
$\phi\:\check\Delta_f\to\Delta_M$.
Thus, taking into account Theorem \ref{2.2}(b), it remains to show that $s_{\phi,\tau}$ 
extends to an equivariant map $D(\phi^*\tau)\to S(\tau\oplus\eps^k)$ lying over $\phi$
if and only if $\check\Delta_f*S^{m-1}_\circ$ admits an equivariant map to $S^{m+k-1}_\circ$.

Let $\nu$ be a normal PL $m$-disc bundle of $M$, regarded as an equivariant PL disc 
bundle over $\Delta_M$ (with fiber $D^m_\circ$), so that $\tau\oplus\nu$ is 
equivariantly isomorphic to the trivial PL $2m$-disc bundle $\eps^{2m}$.
If $s_{\phi,\tau}$ extends to an equivariant map $D(\phi^*\tau)\to S(\tau\oplus\eps^k)$ 
lying over $\phi$, then $s_{\phi,\tau\oplus\nu}$ extends to an equivariant map 
$D\big(\phi^*(\tau\oplus\nu)\big)\to S(\tau\oplus\nu\oplus\eps^k)$ lying over $\phi$
(by considering each fiber of $\tau\oplus\nu$ as the join of $D^m_\circ$ and 
$S^{m-1}_\circ$).
But the latter is equivalent to the existence of an equivariant map
$\check\Delta_f*S^{2m-1}_\circ\to S^{2m+k-1}_\circ$.
Since $\dim\Delta_f\le 2n-m$ and $2n\le 2(m+k)-3$ (using that $2(m+k)\ge 3(n+1)$),
the latter is in turn equivalent to the existence of an equivariant map
$\check\Delta_f*S^{m-1}_\circ\to S^{m+k-1}_\circ$ by Lemma \ref{2.4}(a) below.

Conversely, suppose that there exists an equivariant map
$\check\Delta_f*S^{m-1}_\circ\to S^{m+k-1}_\circ$.
Then $s_{\phi,\eps^m}$ extends to an equivariant map $D(\phi^*\eps^m)\to S(\eps^{m+k})$
lying over $\phi$.
Therefore $s_{\phi,\tau\oplus\eps^m}$ extends to an equivariant map 
$D\big(\phi^*(\tau\oplus\eps^m)\big)\to S(\tau\oplus\eps^{m+k})$ lying over $\phi$
(by considering each fiber of $\tau\oplus\eps^m$ as the join of $S^{m-1}_\circ$ and 
$D^m_\circ$).
By Lemma \ref{2.4}(b) below $s_{\phi,\tau\oplus\eps^i}$ extends to an equivariant map 
$D\big(\phi^*(\tau\oplus\eps^i)\big)\to S(\tau\oplus\eps^{i+k})$ lying over $\phi$
for each $i=m-1,\dots,0$, due to $2n\le 2(m+k)-3$.

In the case where $M$ is not smoothable, we can use the equivariant normal block bundle of
$\Delta_M$ in $M\x M$ in place of $\tau$ (see \cite{RS1}, \cite{BRS}) and the normal block
bundle of some PL immersion of $M$ in $\R^{2m}$ in place of $\nu$.
The proof is similar, except that we can no longer use the homotopy lifting property.
Because of this we have to keep track of additional homotopies, but it is straightforward; 
we leave the details to the reader. 
\end{proof}

\begin{lemma}\label{2.4} (a) {\rm (Conner--Floyd \cite{CF})}
If $\Z/2$ acts freely on a polyhedron $K^k$, the suspension map 
$[K,\,S^{q-1}_\circ]_{\Z/2}\xr{\Sigma} [S^0_\circ*K,\,S^q_\circ]_{\Z/2}$
is onto for $k\le 2q-3$ and one-to-one for $k\le 2q-4$.
\smallskip

(b) Let $\pi\:K\to P$ be a $\Z/2$-equivariant PL map between polyhedra, where the action 
on $K^k$ is free, and let $\Sigma_\pi(K)$ be the double mapping
cylinder of $\pi$, i.e.\ the adjunction space of the partial map 
$K\x[-1,1]\supset K\x\{-1,1\}\xr{f\sqcup f}P\sqcup P$,
with $\Z/2$ acting antipodally on $[-1,1]$.
Let $L$ be an invariant subpolyhedron of $K$ and let $\Sigma_\pi(L)$ be the double
mapping cylinder of $\pi|_L$ (in particular, if $L=\emptyset$, it is $P\sqcup P$).
Let $M$ be a polyhedron with the trivial action of $\Z/2$ and let 
$\phi\:P\to M$ be an equivariant map.
Let $\tau$ be a vector $q$-bundle over $M$ and let $\eps$ be the trivial line bundle 
$M\x\R\to M$, where $\Z/2$ acts antipodally on the fibers of $\tau$ and $\eps$.
Let $S(\tau)$ denote the associated sphere bundle (in particular, $S(\tau\oplus\eps)$ is 
the double mapping cylinder of the bundle projection $\Pi\:S(\tau)\to M$), let 
$\lambda\:L\to S(\tau)$ be an equivariant map lying over $\phi$, and let 
$\Sigma\lambda\:\Sigma_\pi(L)\to S(\tau\oplus\epsilon)$ be the ``double mapping cylinder''
of $\lambda$.
Let $[K,\,S(\tau);\,\lambda]^\phi_{\Z/2}$ denote the set of equivariant extensions 
$K\to S(\tau)$ of $\lambda$ lying over $\phi$ up to homotopy through such extensions.
Then $$[K,\,S(\tau);\,\lambda]^\phi_{\Z/2}\xr{\Sigma} 
[\Sigma_\pi(K),\,S(\tau\oplus\eps);\,\Sigma\lambda]^\phi_{\Z/2}$$
is onto if $k\le 2q-3$ and one-to-one if $k\le 2q-4$.
\end{lemma}

Conner and Floyd require $q\ge 3$ in (a), but this is not needed as shown, for instance, by 
the usual geometric arguments.

Namely, Lemma \ref{2.4}(a) can be proved by a version of Pontryagin's proof of the Freudenthal
suspension theorem \cite{Po}. 
In more detail, the assertion follows by a general position argument when the homotopy 
sets in question are rewritten in their geometric form.
Such a geometric description is given by the absolute case ($L=\emptyset$) of Lemma \ref{3.2} 
below along with the Pontrjagin--Thom construction \ref{4.0}(a).

Alternatively, Lemma \ref{2.4}(a) can be viewed as a special case (with $P=M=pt$ and 
$L=\emptyset$) of Lemma \ref{2.4}(b), which we now prove by adapting the geometric proof of 
the Freudenthal suspension theorem in \cite{FFG} (see also \cite{M1}*{Lemma 7.7(c)}).

\begin{proof}[Proof of (b)] We will prove the surjectivity; the injectivity is proved similarly.

We are given an equivariant extension $\kappa\:\Sigma_\pi(K)\to S(\tau\oplus\eps)$
of $\Sigma\lambda$ lying over $\phi$.
Let $M_+\sqcup M_-$ be the two copies of $M$ in the double mapping cylinder 
$\Sigma_\Pi\big(S(\tau)\big)=S(\tau\oplus\eps)$.
We may assume that $\kappa|_{K\x(-1,1)}$ is PL transverse to $M_+\sqcup M_-$ (see \cite{BRS}).
Let $Q_+=\kappa^{-1}(M_+)$ and $Q_-=\kappa^{-1}(M_-)$.
Since $\Sigma_\pi(L)$ contains the two copies $P_+\sqcup P_-$ of $P$, we have
$P_+\subset Q_+$ and $P_-\subset Q_-$.
Thus by tranversality, $Q_+\but P_+$ and $Q_-\but P_-$ have disjoint product neighborhoods 
in $K\x(-1,1)$, equivariantly PL homeomorphic to their products with $D^q_\circ$.
Since $k\le 2q-3$, after an arbitrarily small equivariant isotopy of $\Sigma_\pi(K)$ keeping
$\Sigma_\pi(L)$ fixed we may assume that $Q_+\but P_+$ and $Q_-\but P_-$ have disjoint 
images under the projection $K\x(-1,1)\to K$.
Then after an equivariant isotopy of $\Sigma_\pi(K)$ rel $\Sigma_\pi(L)$ which restricts 
to a vertical isotopy of $K\x(-1,1)$ (lying over the identity on $K$) we may assume that 
$Q_+\but P_+$ lies in $K\x(0,1)$ and $Q_-\but P_-$ lies in $K\x (-1,0)$.
Hence $\kappa$ is equivariantly homotopic keeping $\Sigma_\pi(L)$ fixed to a map $\kappa_1$
such that $\kappa_1^{-1}(M_+)\subset K\x(0,1)$ and $\kappa_1^{-1}(M_-)\subset K\x(-1,0)$.

Then $\kappa_1(K\x\{0\})$ is disjoint from $M_+\sqcup M_-$, and hence also from some
disjoint neighborhoods $N_+$ and $N_-$ of $M_+$ and $M_-$.
Since $\id_{S(\tau\oplus\eps)}$ is equivariantly homotopic to a map sending
$N_+$ onto $M_+\cup S(\tau)\x[0,1)$ and $N_-$ onto $M_-\cup S(\tau)\x(-1,0]$,
$\kappa_1$ is equivariantly homotopic to a map $\kappa_2$ sending
$K\x\{0\}$ into $M\x\{0\}$, $P_+\cup K\x[0,1)$ into $M_+\cup S(\tau)\x[0,1)$
and $P_-\cup K\x(-1,0]$ into $M_-\cup S(\tau)\x(-1,0]$.
With some work, this homotopy can be amended so as to fix $\Sigma_\pi(L)$.
Since $\kappa_2$ is equivariantly homotopic rel $\Sigma_\pi(L)$ to $\kappa$, which 
lies over $\phi$, and $\Pi$ is an equivariant fibration, 
$\kappa_2|_{K\x\{0\}}\:K\x\{0\}\to M\x\{0\}$ is equivariantly homotopic rel $\Sigma_\pi(L)$ 
to a map which lies over $\phi$.
By similar arguments, $\kappa_2$ is equivariantly homotopic rel $\Sigma_\pi(L)$ 
to a map $\kappa_3$ which lies over $\phi$ and sends $K\x\{0\}$ into $M\x\{0\}$, 
$P_+\cup K\x[0,1)$ into $M_+\cup S(\tau)\x[0,1)$ and 
$P_-\cup K\x(-1,0]$ into $M_-\cup S(\tau)\x(-1,0]$.
Then by the fiberwise Alexander trick (performed by induction over simplexes of $K$, in
an order of increasing dimension) $\kappa_3$ is equivariantly homotopic 
rel $\Sigma_\pi(L)$ and over $\phi$ to the ``double mapping cylinder'' of a map 
$K\to S(\tau)$ (which must extend $\lambda$ and lie over $\phi$). 
\end{proof}

By combining Theorem \ref{2.1}, Theorem \ref{2.3} and Lemma \ref{2.4}(a), we obtain

\begin{corollary*}[=Theorem \ref{th2}] Let $N^n$ be a compact smooth manifold 
(resp.\ a compact polyhedron), $M^m$ a smooth (PL) manifold and $f\:N\to M$ a stable smooth (PL) map, 
where $m\ge n$ and $2n-m\le 2k-3$.
In the smooth case, assume additionally that either $f$ is a fold map or $3n-2m\le k$. 
Then $f$ is $k$-realizable if and only if it is a smooth (PL) $k$-prem.
\end{corollary*}

\subsection{Proofs of Theorems \ref{th1} and \ref{th3}}

\begin{example}\label{2.6} Let $f\:S^3\to S^2$ be a PL map with Hopf invariant
$H(f)$ even but nonzero.
Consider $K:=\Cone(f)\cup_{S^2}\Cone(-f)$, where $-$ denotes the antipodal
involution on $S^2$.
This $K$ is endowed with the involution interchanging the two
copies of the mapping cone of $f$.
Then $S^0_\circ*K$ is $\Z/2$-homotopy equivalent to
$\Cone(\Sigma f)\cup_{S^3}\Cone(-\Sigma f)$.
Since $\Sigma f\:S^4\to S^3$ is null-homotopic, there exists an equivariant map
$S^0_\circ*K\to S^3_\circ$.

At the same time, there exists no equivariant map $K\to S^2_\circ$.
Indeed, every equivariant map $\phi\:S^2_\circ\to S^2_\circ$ has an odd degree
since its first obstruction to homotopy with $\id_{S^2}$ is on the one hand
$\deg(\phi)-1\in\Z\simeq H^2(S^2)$ and on the other hand the image of
the first obstruction to equivariant homotopy with $\id_{S^2}$ in
$H^2_{\Z/2}(S^2_\circ)\simeq\Z$ under the forgetful map, which is onto $2\Z$.
Now $H(\phi f)=\deg(\phi)^2H(f)\ne 0$, which is easy to see
from the definition of the Hopf invariant as the total linking number between
the point-inverses of two regular values. 
\end{example}

A {\it $k\lambda$-framing} of a bundle $\xi$ over a base $X$ is an (unstable)
isomorphism $\xi\simeq k\lambda$ with $k$ copies of the line bundle $\lambda$
over $X$.
A smooth manifold is called {\it stably $k\lambda$-parallelizable} if it
admits a $k$-dimensional normal bundle isomorphic to $k\lambda$.

Recall from \cite{KS}*{\S1} that if $f\:S^n\imm\R^m$ is a smooth immersion
whose normal bundle is trivial over its non-one-to-one points
$\Delta_f\imm S^n$, then the normal bundle of $\Delta_f/t\imm\R^m$ is
$(m-n)\lambda$-framable by the line bundle $\lambda$ associated to
the double covering $\Delta_f\to\Delta_f/t$.
Specifically, the normal bundle of $\Delta_f\to\Delta_f/t\imm\R^m$ is
equivariantly diffeomorphic to $\Delta_f\x\bigoplus_{m-n}\R[\Z/2]$ endowed
with the diagonal action of $\Z/2$, but each copy of $\R[\Z/2]$ splits into
the direct sum of $\R$ with the trivial action and $\R$ with the sign action
of $\Z/2$.

This corresponds to the $m\lambda$-framing of the normal bundle of
$\Delta_f/t$ in $\tilde S^n/t$, which is implicit in the proof of
Theorem \ref{2.3}, under an $n\lambda$-framing of the tangent bundle of
$\tilde S^n/t$, constructed in \cite{A2}*{Lemma 3.1}.

\begin{theorem}[Realization Principle] \label{2.7} Let $\lambda$ be the line
bundle associated to a double covering $\bar M\to M$ over a closed smooth
$(m-n)\lambda$-parallelizable manifold of dimension $2n-m$.
If $2m\ge 3(n+1)$, there exists a stable smooth immersion
$f\:S^n\imm\R^m$ such that $\Delta_f$ is equivariantly diffeomorphic to $\bar M$.
\end{theorem}

It is very easy to construct a stable immersion $\phi$ of {\it some} closed
stably parallelizable $n$-manifold into $\R^m$ with trivial normal bundle
such that $\Delta_\phi$ is equivariantly diffeomorphic to
$\bar M\sqcup\bar M$.
Specifically $\phi$ is the double of the ``figure 8'' proper immersion
$\bar M\x D^{n-m}\imm\R^{m-1}\x [0,\infty)$, depicted in \cite{KS}*{\S4}.

\begin{proof} By Koschorke's framed version \cite{Ko}*{Proof of Theorem 1.15}
of the Whitney--Haefliger trick
\cite{Hae}*{D\'emonstration du Th\'eor\`eme 2, a}, \cite{Ad}*{\S VII.4} there exists a
self-transverse regular homotopy $F\:S^{n-1}\x I\imm\R^{m-1}\x I$ between
the standard embedding $g\:S^{n-1}\emb\R^{m-1}$ and some embedding
$g'\:S^{n-1}\emb\R^{m-1}$ such that the conclusion of Theorem \ref{2.7} holds for $F$
in place of $f$.
(Koschorke deals with a slightly more general situation: his $g$ is a
self-transverse immersion, whose double points he wants to eliminate by
a regular homotopy, and his $M$ is a $(m-n)\lambda$-framed null-bordism of
$\Delta_g/t$.)
Now $g'$ is smoothly isotopic to the standard embedding (see \cite{Hae}),
so $F$ can be capped off to a stable smooth immersion $f\:S^n\imm\R^m$
without adding new double points. 
\end{proof}

\begin{example*}[=Theorem \ref{th1}] For each $n=4k+3\ge 15$ there exists a stable smooth
immersion $S^n\imm\R^{2n-7}$ that is $3$-realizable, but is not a $3$-prem.
\end{example*}

\begin{proof} Let $K$ be as in Example \ref{2.6}, with the involution $t$.
Then $K/t$ is the mapping cone of the composition $S^3\xr{f}S^2\xr{p}\RP^2$,
where $p$ is the double covering.
The cylinder of $pf$ properly embeds, via the graph of $pf$, into
$\RP^2\x D^4$.
Hence $K$ minus two symmetric small $4$-balls properly equivariantly embeds
into $S^2_\circ\x D^4$, where $D^4$ has the trivial action of $\Z/2$.
Therefore $K$ equivariantly embeds into $S^3_\circ\x D^4$.
Consider an equivariant regular neighborhood $R$ of $K$ in $S^3_\circ\x D^4$,
and let $M$ be its double $M=\partial(R\x I)$.
Since $M$ equivariantly retracts onto a copy of $K$, by Example \ref{2.6} it does
not admit an equivariant map to $S^2_\circ$, but $S^0_\circ*M$ admits an
equivariant map to $S^3_\circ$.
Since $M/t=\partial((R/t)\x I)$ embeds with trivial line normal bundle into
the parallelizable manifold $\RP^3\x D^5$, it is stably parallelizable.
Since the tangent bundle of $\RP^k$ is stably equivalent to the sum of $k+1$
copies of the canonical line bundle $\gamma$, the bundle $4\gamma$ is stably
trivial over $\RP^3$.
Hence $M/t$ admits a skew $4k\lambda$-framed normal bundle for each $k\ge 0$,
where $\lambda$ is associated with the double covering $M\to M/t$ and so
is the pullback of $\gamma$.
Thus by Theorem \ref{2.7}, for each $n=4k+3\ge 15$, $M$ is the double point locus
of some stable smooth immersion $S^n\imm\R^{2n-7}$. 
\end{proof}

\begin{theorem*}[=Theorem \ref{th3}] Let $N^n$ be a compact smooth manifold
(resp.\ a compact polyhedron), $M^m$ a smooth (PL) manifold and suppose that $2m>3n$.
A stable smooth (PL) map $f\:N\to M$ is $1$-realizable if and only if it is a smooth (PL) $1$-prem.
\end{theorem*}

\begin{proof} If $f$ is $1$-realizable, by the proof of Theorem \ref{2.3}
there exists a stable equivariant map $\Delta_f\to S^0_\circ$, i.e.\
an equivariant map $S^\infty_\circ*\Delta_f\to S^{\infty+1}_\circ$.
Consider the Euler class $e(\lambda)\in H^1(\Delta_f/t;\Z_\lambda)$ of
the line bundle $\lambda$ associated to the double covering
$\Delta_f\to\Delta_f/t$, where $\Z_\lambda$ is the integral local coefficient
system associated with $\lambda$.
The Euler class is cohomological, hence in particular a stable invariant due
to the natural isomorphism
$H^i(S^0_\circ*K;\,P)\simeq H^{i-1}(K;\,P\otimes\Z_\lambda)$ for any local
coefficient system $P$.
So $e(\lambda)=0$, whence $\lambda$ is orientable and so $\Delta_f$ admits
an equivariant map to $S^0$.
By Theorem \ref{2.1}, $f$ is a $1$-prem. 
\end{proof}

\section{Equivariant stable cohomotopy} \label{s3}

Suppose that $P$ is a pointed polyhedron and $G$ is a finite group acting
on $P$ and fixing the basepoint $*$.
We also assume that $P$ is $G$-homotopy equivalent to a compact polyhedron.
If $V$ is a finite-dimensional $\R G$-module, let $S^V$ be the one-point
compactification of the Euclidean space $V$ with the obvious action of $G$.
If $G$ acts trivially on $V$, we follow a standard convention of denoting
$V$ by the integer $\dim V$.
The equivariant stable cohomotopy group
$$\omega^{V-W}_G(P):=[S^{W+V_\infty}\wedge P, S^{V+V_\infty}]^*_G$$ is
well-defined, where $V_\infty$ denotes a sufficiently large (with respect
to the partial ordering with respect to inclusion) finite-dimensional
$\R G$-submodule of the countable direct sum $\R G\oplus\R G\oplus\dots$
(see \cite{M+}, \cite{Ada}).

Our main interest here lies in the case where $G=\Z/2$, $V$ is the sum of
$m$ copies of the nontrivial one-dimensional representation $T$ of $\Z/2$
(i.e.\ $V$ is $\R^m$ with the sign action of $\Z/2$) and $P=K_+$ (i.e.\
the union of $K$ with a disjoint basepoint), where the action of $\Z/2$
on $K$ is free.
In this case, the following lemma guarantees that $V_\infty$ from the
above definition can be taken to be any $\R[\Z/2]$-module of sufficiently
large dimension.

\begin{lemma}[Bredon--Hauschild; see \cite{M+}*{IX.I.4}, \cite{Ada}*{3.3}] \label{3.1}
The basepoint preserving homotopy set $[P,S^{mT}]^*_{\Z/2}$ surjects to
$\omega^{mT}_{\Z/2}(P)$ if $\dim P\le 2m-1$ and $\dim P^{\Z/2}\le m-1$, and
injects there if $\dim P\le 2m-2$ and $\dim P^{\Z/2}\le m-2$.
\end{lemma}

In the above mentioned case $P=K_+$, the lemma follows by a general position
argument when the Pontrjagin--Thom construction \ref{4.0}(a) is taken into account.
\medskip

As before, by $S^k_\circ$ we denote the $k$-sphere with the antipodal
involution, i.e.\ the unit sphere in $(k+1)T$.
For $k\ge m$, we have that $S^k_\circ\but S^{m-1}_\circ$ is
$\Z/2$-homeomorphic to $S^{k-m}_\circ\x mT$.
Shrinking to points $S^{m-1}_\circ$ and each fiber $S^{k-m}_\circ\x\{pt\}$,
we get an equivariant map $\rho^k_m\:S^k_\circ\to S^{mT}$.

For any $k$-polyhedron $K$ with a free $\Z/2$ action and a $\Z/2$-invariant
subpolyhedron $L\subset K$, any equivariant map $\ell\:L\to S^{m-1}_\circ$ extends
to an equivariant map $\phi_K^\ell\:K\to S^\infty_\circ$, which is unique up
to equivariant homotopy $\rel L$.
Here $\infty$ may be thought of as a sufficiently large natural number
(specifically, $k+1$ will do).
If $f,g\:K\to S^{m-1}_\circ$ are equivariant extensions of $\ell$, they are
joined by a $\rel L$ equivariant homotopy
$\phi^\ell_{f,g}\:K\x I\to S^\infty_\circ$, which is unique up to equivariant
homotopy $\rel K\x\partial I\cup L\x I$.

\begin{lemma}[Melikhov \cite{M3}] \label{3.2} Let $\Z/2$ act freely on a $k$-polyhedron
$K$ and trivially on $I$.
Let $L$ be a $\Z/2$-invariant subpolyhedron of $K$ and
$\ell\:L\to S^{m-1}_\circ$ an equivariant map.
Suppose that $k\le 2m-3$ for (a), (c) and $k\le 2m-4$ for (b).

(a) $\ell$ extends to an equivariant map $K\to S^{m-1}_\circ$ if and only if
$\rho^\infty_m\phi_K^\ell\:K\to S^{mT}$ is equivariantly null-homotopic
$\rel L$.

(b) Equivariant extensions $f,g\:K\to S^{m-1}_\circ$ of $\ell$ are
equivariantly homotopic $\rel L$ if and only if
$\rho^\infty_m\phi^\ell_{f,g}\:K\x I\to S^{mT}$ is equivariantly
null-homotopic $\rel K\x\partial I\cup L\x I$.

(c) For any equivariant extension $f\:K\to S^{m-1}_\circ$ of $\ell$ and any
equivariant $\rel L$ self-homotopy $H\:K\x I\to S^{mT}$ of the constant map
$K\to *\subset S^{mT}$ there exists an equivariant extension
$g\:K\to S^{m-1}_\circ$ of $\ell$ such that $\rho^\infty_m\phi^\ell_{f,g}$
and $H$ are equivariantly homotopic $\rel K\x\partial I\cup L\x I$.
\end{lemma}

\begin{example} Lemma \ref{3.2}(a) with $L=\emptyset$ does not hold for $k=2m-2=4$.

Indeed, let $K$ be the $4$-dimensional $\Z/2$-polyhedron of Example \ref{2.6}.
Thus there exists an equivariant map $f\:S^0_\circ*K\to S^3_\circ$.
Then by the trivial impliciation in Lemma \ref{3.2}(a), the composition 
$S^0_\circ*K\xr{\phi_{S^0_\circ*K}}S^\infty_\circ\xr{\rho^\infty_4} S^{4T}$ 
is equivariantly null-homotopic.
Consequently, the composition 
$S^0_\circ*K\xr{S^0_\circ*\phi_K}S^0_\circ*S^\infty_\circ\xr{\rho^{\infty+1}_4} S^{4T}$ 
is also equivariantly null-homotopic by a homotopy $H$.
The latter composition sends $S^0_\circ$ into $S^{4T}\but\R^{4T}=\{\infty\}$, and we 
may assume that $H$ sends $S^0_\circ\x I$ into $\infty$ (otherwise it can be amended 
by shrinking the loop $H(S^0_\circ\x I)$).
Since the diagram
\[\begin{CD}
S^0_\circ*K@>S^0*\phi_K>>S^0*S^\infty_\circ@>\rho^{\infty+1}_4>>S^{4T}\\
@VVV@VVV@|\\
S^T\vee K_+@>S^T\vee\phi_K>>S^T\vee (S^\infty_\circ)_+@>S^T\vee\rho^\infty_3>>
S^T\wedge S^{3T}
\end{CD}\]
commutes, we get an equivariant null-homotopy of the bottom line.
Hence by Lemma \ref{3.1} the composition $K\xr{\phi_K}S^\infty_\circ\xr{\rho^\infty_3}S^{3T}$
is equivariantly null-homotopic.
If Lemma \ref{3.2}(a) with $L=\emptyset$ were true for $k=2m-2=4$, this would imply
that there exists an equivariant map $K\to S^2_\circ$, contradicting Example \ref{2.6}.
\end{example}

\subsection{Proof of Theorem \ref{th4}}

Given a compact $n$-polyhedron $X$ and a map $f\:X\to M$ into a PL $m$-manifold, we let
$$\Theta(f):=[\rho^\infty_m\phi_{\tl f}]\in
[\tl X_+,S^{mT}]^{\rel\tl X_+\but\Delta_f}_{\Z/2}=
\omega^{mT}_{\Z/2}(\tl X_+,\tl X_+\but\Delta_f).$$

\begin{theorem}\label{3.3} Let $X^n$ be a compact polyhedron, $M^m$ a PL manifold and 
let $f\:X\to M$ be a stable PL map.
Suppose that $2(m+k)\ge 3(n+1)$ in (a), (c) and $2(m+k)>3(n+1)$ in (b).

(a) $f$ is $k$-realizable if and only if
$\Theta(f)=0\in\omega^{kT}_{\Z/2}({\Delta_f}_+)$.

(b) $k$-realizations $g,g'\:X\emb M\x\R^k$ are isotopic through
$k$-realizations of $f$ if and only if
$\Theta_f(g,g')=0\in\omega^{kT-1}_{\Z/2}({\Delta_f}_+)$.

(c) If $g$ is a $k$-realization of $f$, for each
$\alpha\in\omega^{kT-1}_{\Z/2}({\Delta_f}_+)$ there exists a $k$-realization
$g'$ of $f$ such that $\Theta_f(g,g')=\alpha$.
\end{theorem}

\begin{proof} Let $R$ be as in the proof of Theorem \ref{2.3}, so
$$\omega^{(m+k)T}_{\Z/2}(\tl X_+,\tl X_+\but\Delta_f)\simeq
\omega^{(m+k)T}_{\Z/2}(\tl X_+,\Cl{\check X_+\but R}).$$
By the proof of Theorem \ref{2.3}, $R/\Fr R$ is equivariantly homotopy equivalent
with ${\Delta_f}_+\wedge S^{mT}$.
Therefore
$$\omega^{(m+k)T}_{\Z/2}(\tl X_+,\Cl{\check X_+\but R})\simeq
\omega^{(m+k)T}_{\Z/2}(R_+,{\Fr R}_+)\simeq\omega^{kT}_{\Z/2}({\Delta_f}_+).
$$
Assertion (a) now follows from Theorem \ref{2.2} and Lemma \ref{3.2}(a).
Similarly (b) and (c) follow (see \cite{M3} for some details) from
the ``moreover'' part of Theorem \ref{2.2} and its parametric version below
(Theorem \ref{3.4}) using Lemma \ref{3.2}(b,c). 
\end{proof}

\begin{theorem}\label{3.4} {\rm (Melikhov \cite{M1}*{7.9(a)})} Let $X^n$ be a compact polyhedron, 
$M^m$ a PL manifold and $f\:X\to M$ a PL map, where $2m>3(n+1)$.
If embeddings $g,g'\:X\emb M$ are $C^0$-close to $f$, they are isotopic
through $C^0$-approximations of $f$ if and only if $g\x g$ and $g'\x g'$
are isovariantly homotopic through $C^0$-approximations of $f\x f$.
\end{theorem}

By combining Theorem \ref{3.3}(a,b), Lemma \ref{3.1} (applied to ${\Delta_f}_+$ and
${\Delta_f}_+\wedge S^1$), the absolute case ($L=\emptyset$) of
Lemma \ref{3.2}(a,c), and Theorem \ref{2.1}, we get a slightly different proof of
Theorem \ref{th2}, and also its relative version

\begin{theorem}\label{3.5} Let $N^n$ be a compact smooth manifold
(resp.\ a compact polyhedron), $M^m$ a smooth (PL) manifold and $f\:N\to M$ 
a stable smooth (PL) map, where $m\ge n$, $2n-m\le 2k-3$, and also $3n-2m<2k-3$.
In the smooth case, assume additionally that either $f$ is a fold map or $3n-2m\le k$.
Then every $k$-realization of $f$ is isotopic through $k$-realizations
of $f$ to a smooth (PL) vertical lift of $f$.
\end{theorem}

\begin{example*}[=Theorem \ref{th4}] For $n\ge 7$ there exists a stable smooth immersion
$f\:S^n\imm\R^{2n-2}$ and its $2$-realization $g\:S^n\emb\R^{2n}$,
not isotopic through $2$-realizations of $f$ to any vertical lift of $f$.
\end{example*}

\begin{proof} Let $f$ be such that $\Delta_f$ is $\Z/2$-homeomorphic to
$S^2\x S^0_\circ$ (cf.\ Theorem \ref{2.7}).
Clearly, $f$ lifts to an embedding $g_0\:S^n\emb\R^{2n-1}\subset\R^{2n}$ (cf.\ Remark \ref{2.1'}).
Let $\alpha$ be the generator of
$\omega^{2T-1}_{\Z/2}({\Delta_f}_+)\simeq\pi_{2+\infty}(S^{1+\infty})
\simeq\Z/2$.
By Theorem \ref{3.3}(c), there exists a $2$-realization $g$ of $f$ with
$\Theta_f(g_0,g)=\alpha$.
If $g_1\:S^n\emb\R^{2n}$ is an embedding, projecting onto $f$, then
$\tl g_1|_{\Delta_f}\:\Delta_f\to S^1_\circ$ is equivariantly homotopic
to $\tl g_0$ due to $\pi_2(S^1)=0$.
Hence $\Theta_f(g_1,g_0)=0$ and $\Theta_f(g_1,g)=\alpha$, so $g$ cannot be
isotopic to any such $g_1$ through $2$-realizations of $f$. 
\end{proof}

\subsection{Proof of Theorem \ref{th5}}

The following lemma improves on the absolute case ($L=\emptyset$) of
Lemma \ref{3.2}(a).
As shown by the proof of Theorem \ref{th4}, the analogous strengthening of
the absolute case of Lemma \ref{3.2}(c) is not true.

\begin{lemma}\label{3.7} Let $K$ be a polyhedron with a free action of $\Z/2$,
and suppose that $m=2,4$ or $8$.
In the case $m=8$ assume additionally $\dim K\le 22$.
There exists an equivariant map $K\to S^{m-1}_\circ$ if and only if the composition
$K\xr{\phi_K}S^\infty_\circ\xr{\rho^\infty_m}S^{mT}$ is equivariantly
null-homotopic.
\end{lemma}

\begin{proof} $\rho^\infty_m$ is the composition of the quotient map
$q$ of $S^\infty_\circ$ onto the ``cosphere''
$S_m:=S^\infty_\circ/S^{m-1}_\circ$ and a fibration $S_m\to S^{mT}$ (with all
fibers $S^\infty$ except for one singleton fiber).
Hence the composition $K\xr{\phi_K}S^\infty_\circ\xr{q}S_m$ is equivariantly
null-homotopic.
In the case $m=8$ the additional hypothesis allows to assume that
the null-homotopy of $q\phi_K$ lies in $S^{23}_\circ/S^7_\circ$.
Now consider the diagram
\[\begin{CD}
S^{m-1}_\circ @.\quad\subset\quad@.S^\infty_\circ @>q>>S_m\\
@.@.@V{h}VV\hskip -55pt\swarrow\\
@.@.\KP^\infty
\end{CD}\]
where $\KP^\infty$ denotes $\CP^\infty$ if $m=2$, the quaternionic infinite
projective space $QP^\infty$ if $m=4$ and the octonionic (also known as
Cayley) projective plane $OP^2$ if $m=8$, and $h$ is the standard Hopf bundle
with fiber $S^{m-1}$ (regarded as a partial map defined on
$S^{23}\subset S^\infty$ in the case $m=8$).
Since $h$ is equivariant with respect to the antipodal involution on
$S^\infty$ and the identity on $\KP^\infty$, we may identify
the distinguished $S^{m-1}_\circ\subset S^\infty_\circ$ with one of its fibers.
Then $h$ factors through $S_m$, so the composition
$K\xr{\phi_K}S^\infty\xr{q}S_m\xr{}\KP^\infty$ is null-homotopic.
Since $h$ is a fibration, this null-homotopy lifts to an equivariant
homotopy of $\phi_K$ to a map $K\to S^{m-1}$. 
\end{proof}

Lemma \ref{3.7} fits to replace the absolute case of Lemma \ref{3.2}(a) in the above
alternative proof of Theorem \ref{th2}, thus proving

\begin{theorem*}[=Theorem \ref{th5}] Let $N^n$ be a compact smooth manifold (resp.\
a compact polyhedron), $M^{2n-2k+2}$ a smooth (PL) manifold and $f\:N\to M$ 
a stable smooth (PL) map, where $n\ge 2k-1$.
In the smooth case, assume additionally that either $f$ is a fold map or $n\ge 3k-4$.
If $f$ is $k$-realizable and $k\in\{2,4,8\}$, then $f$ is a smooth (PL) $k$-prem.
\end{theorem*}

\section{Skew-co-oriented skew-framed comanifolds} \label{s4}

In this section we use geometric methods to analyze failure of injectivity
of the map $[K_+,S^{mT}]^*_{\Z/2}\to\omega^{mT}_{\Z/2}(K_+)$ from Lemma \ref{3.1}
beyond the stable range.

Let $\bar X$ be a compact polyhedron with a free PL action of
$\Z/2=\left<t\mid t^2\right>$.
Write $X=\bar X/t$, and let $\lambda$ be the line bundle associated with
the double covering $\bar X\to X$.
Given an equivariant PL map $f\:\bar X\to S^{qT+r}$, after an equivariant
homotopy we may assume that it is PL transverse to the origin $O$ of
the Euclidean space $qT+r\subset S^{qT+r}$, and so a $\Z/2$-invariant regular
neighborhood of $\bar Q:=f^{-1}(O)$ in the polyhedron $\bar X$ is
equivariantly PL homeomorphic to $Q\x I^{q+r}$ \cite{RS1},
\cite{BRS}*{\S II.4}.
A fixed orientation of $S^{qT+r}$ induces an orientation of this bundle,
and hence a {\it $\lambda^{\otimes q}$-co-orientation} of $Q:=\bar Q/t$
in $X$, that is an isomorphism
$H^{q+k}(X,X\but Q;\Z_{\lambda^{\otimes q}})\simeq\Z$,
where $\Z_\lambda$ denotes the integral local coefficient system associated
with $\lambda$ (compare \cite{BRS}*{\S IV.1}).
The action of $\Z/2$ on $S^{qT+r}$ fixes $O$ and each vector of a fixed
$r$-frame at $O$, and inverts each vector of a fixed complementary $q$-frame.
Hence $Q$ has a normal PL disc bundle $\xi$ in $X$, endowed with
a {\it $(q\lambda+r)$-framing} (as a $\lambda^{\otimes q}$-oriented bundle),
i.e.\ a $\lambda^{\otimes q}$-orientation preserving isomorphism
$\xi\simeq q\lambda|_Q\oplus r\eps$ of PL disc bundles, where $\eps$ denotes
the trivial line bundle.
This isomorphism also makes sense when $q$ and $r$ are not necessarily
nonnegative by transferring any negative terms to the left hand side of
the equation.

For brevity, we shall call a $\lambda^{\otimes q}$-co-oriented subpolyhedron
$Q\subset X$ with a $(q\lambda+r)$-framed normal PL disc bundle in $X$ a {\it
$(q\lambda+r)$-comanifold} in $X$.
A {\it $(q\lambda+r)$-cobordism} between two $(q\lambda+r)$-comanifolds
$Q_0$, $Q_1$ in $X$ is a $(q\lambda+r)$-comanifold in $X\x I$ meeting
$X\x\{i\}$ in $Q_i$ for $i=0,1$.
(For the reader who wonders what object we might call just ``comanifold'',
it is an embedded mock bundle in the sense of \cite{BRS}*{p.\ 34}, or
equivalently a subpolyhedron of $X$ that has a normal block bundle in $X$.)

The usual Pontryagin--Thom argument shows that the pointed equivariant
homotopy set $[\bar X,S^{qT+r}]_{\Z/2}$ is in pointed bijection with
the pointed set $\Emb^{q\lambda+r}(X)$ of $(q\lambda+r)$-comanifolds in
$X$ up to $(q\lambda+r)$-cobordism.
Furthermore, it follows that the equivariant stable cohomotopy group
$\omega^{qT+r}_{\Z/2}(\bar X_+)$ is isomorphic to the group
$\Imm^{q\lambda+r}(X)$ of singular $(q\lambda+r)$-comanifolds in $X$
up to singular $(q\lambda+r)$-cobordism.
Here a {\it singular $(q\lambda+r)$-comanifold} in $X$ is the projection
$Q\to X$ of a $(q\lambda+r+\infty)$-comanifold $Q$ in $X\x\R^\infty$
(compare \cite{BRS}*{\S IV.2}).

The dual bordism group $\Omega^{\sfr(k)}_n(X;\lambda)$ consists of
stably $k\lambda$-parallelized $\lambda^{\otimes k}$-oriented singular
PL $n$-manifolds $f\:N\to X$ up to stably $k\lambda$-parallelized
$\lambda^{\otimes k}$-oriented singular bordism (compare \cite{A2}).
Here a {\it $\lambda^{\otimes k}$-orientation} of $N$ is an isomorphism
$H^n(N;\Z_{f^*\lambda^{\otimes k}})\simeq\Z$.
A {\it stable $k\lambda$-parallelization} of $N$ (as
a $\lambda^{\otimes k}$-oriented manifold) is a
$\lambda^{\otimes k}$-orientation preserving isomorphism between
$f^*(k\lambda)$ and a $k$-dimensional normal bundle of $N$.

Let us summarize:

\begin{lemma}\label{4.0} Let $\bar X\to X$ be a double covering of compact
polyhedra, and let $\lambda$ be the associated line bundle.

(a) {\rm (Pontrjagin--Thom construction)}  \cite{RS1}*{3.3},
\cite{BRS}*{IV.2.4} There exist a natural pointed bijection
$[\bar X,S^{qT+r}]_{\Z/2}\leftrightarrow\Emb^{q\lambda+r}(X)$ and a natural
isomorphism $\omega^{qT+r}_{\Z/2}(\bar X_+)\simeq\Imm^{q\lambda+r}(X)$.

(b) {\rm (Poincar\'e duality)} \cite{BRS}*{II.3.2, IV.2.4} If $X$ is
a closed stably $p\lambda$-parallelizable $m$-manifold, there exists
an isomorphism $\Imm^{q\lambda}(X)\simeq\Omega^{\sfr(p+q)}_{m-q}(X;\lambda)$,
which gets natural once $X$ is endowed with
a $\lambda^{\otimes p}$-orientation and a stable $p\lambda$-parallelization.
\end{lemma}

\begin{remark} Let us indicate a relation to a more traditional approach.
Let $\Imm^{q\bigstar}(X)$ be defined similarly to $\Imm^{q\lambda}(X)$,
except that $\lambda$ is not globally defined, but is a part of the data of
the $q\lambda$-comanifold.
It is well-known that this group is isomorphic to
$[S^\infty*X;\, S^\infty*(\RP^\infty/\RP^{q-1})]$, cf.\ \cite{AE}.
Here $\RP^\infty/\RP^{q-1}$ is the Thom space of the bundle $q\gamma$ over
$\RP^{\infty-q}$, so the point-inverse $Q$ of the basepoint of this space is
$q\lambda$-framed in $X$, where $\lambda$ is the pullback of $\gamma$
under the map $Q\to\RP^{\infty-q}$.
\end{remark}

The notation ``Imm'' is partially justified by

\begin{lemma}\label{4.1} Suppose that $q\ge 0$, $r\ge 0$ and $q+r\ge 1$.

(a) {\rm (Hirsch Lemma)} Every element of $\Imm^{q\lambda+r}(X)$ admits
an immersed representative, unique up to immersed $(q\lambda+r)$-cobordism.

(b) {\rm (Compression Theorem)} For large $n$, every embedded
$(q\lambda+r+n)$-comanifold $Q$ in $X\x\R^n$ is isotopic by
an arbitrarily $C^0$-small ambient isotopy to one whose last $n$ vectors
of framing are standard.
The ambient isotopy can be chosen to fix $Y\x\R^n$ where $Y$ is
a subpolyhedron of $X$ such that $Q$ is PL transverse to $Y\x\R^n$ and
the last $n$ vectors of the framing of $Q$ are already standard over
$Q\cap Y\x\R^n$.
\end{lemma}

More specifically, ``large $n$'' means $q+r+n\ge\dim X+1$.
Using \cite{Hae}, this can be weakened to $q+r+n\ge\frac12(\dim X+3)$.

If the ambient isotopy is replaced by regular homotopy, (b) holds in
codimension one \cite{Hi}.

\begin{proof} Clearly (a) follows from (b).

By PL transversality \cite{BRS}*{II.4.4}, given a triangulation of $X$, any
singular $(q\lambda+r)$-comanifold $f\:Q\to X$ can be altered by
an arbitrarily $C^0$-small homotopy so that every simplex $\Delta^i$ of this
triangulation meets $Q$ in a singular $(q\lambda+r)$-comanifold
$f|_{f^{-1}(\Delta^i)}\:f^{-1}(\Delta^i)\to\Delta^i$.
If $f$ is the projection of $Q\subset X\x\R^n$, this homotopy can be chosen to fix
$f^{-1}(Y)$ and lift to an arbitrarily $C^0$-small ambient isotopy of $X\x I$
fixing $Y\x I$.
By \cite{BRS}*{Lemma II.1.2}, a singular $(q\lambda+r)$-comanifold in
the PL manifold $\Delta^i$ is a singular PL $(i-q-r)$-manifold with boundary
$\partial f^{-1}(\Delta^i)=f^{-1}(\partial\Delta^i)$.
Picking a sufficiently fine triangulation of $X$, we may assume that each
such PL manifold $Q\cap (\Delta^i,\partial\Delta^i)\x\R^n$ is contained
in a $(i-q-r)$-PL ball contained in $Q$.
Since $n$ is large, it follows that each such intersection is smoothable
(with corners) as a submanifold in the standard smooth structure on
$\Delta^i\x\R^n$, so that the restriction of the smoothing to the boundary
agrees with those constructed earlier over the faces of $\Delta^i$.
Part (b) now follows by induction on the simplices of $X$ from its smooth
case, which was proved in \cite{RS2}. 
\end{proof}

\begin{addendum}[Dense $h$-principle for PL immersions]
Let $P$ be an $n$-polyhedron, $f\:P\to\R^{2n}$ a PL map, and $g\:P\imm\R^{2n}$
a PL immersion.
Then $f$ is arbitrarily $C^0$-close to a PL immersion, PL regularly homotopic
to $g$.
\end{addendum}

\begin{proof} Let $h\:P\x I\to\R^{2n}\x I$ be a generic homotopy between $g$
and $f$.
By general position, it fails to be an immersion only at finitely many points
$(p,t)$.
Let $L$ be the link of $p$ in $P$.
The restriction of $h$ to $(p*L)\x\{0\}\cup L\x I$ is a regular homotopy
$(p*L)\x\{0\}\cup L\x I\xr{h_p}B^{2n}\cup S^{2n-1}\x I\xr{g_p}\R^{2n}\x I$.
We redefine $h$ on $(p*L)\x I$ by $h(x,s)=g_p((\pi h_p\phi(x,s),s))$, where
$\phi\:(p*L)\x I\to (p*L)\x\{0\}\cup L\x I$ is a collapse and
$\pi\:B^{2n}\x I\to B^{2n}$ is the projection.
Doing this for every cusp $(p,t)$ converts $h$ to a regular homotopy,
which is sufficiently close to $h$ as long as each $(p*L)$ is chosen to be
sufficiently small. 
\end{proof}

Let us recall the geometric definition of the cup product in $\Imm^*(X)$.
Let $\phi\:Q\to X$ be a singular $(q\lambda+r)$-comanifold representing
some $\Phi\in\Imm^{q\lambda+r}(X)$.
The {\it transfer}
$\phi_!\:\Imm^{q'\lambda+r'}(Q)\to\Imm^{(q+q')\lambda+(r+r')}(X)$
of the normal bundle of $Q$ sends a representative $\psi\:M\to Q$ to
the class of $\phi\psi\:M\to Q\to X$, endowed with the skew framing
obtained by combining those of $\phi$ and $\psi$, and the
$\lambda$-co-orientation obtained by combining those of $\phi$ and $\psi$.
For any $\Phi'\in\Imm^{q'\lambda+r'}(X)$, the cup product
$\Phi\Cup\Phi'$, defined originally in terms of the cross product, equals
$\phi_!\phi^*(\Phi')$ \cite{BRS}.

Let $\xi$ be a $q\lambda$-framed PL disc bundle over $X$.
By PL transversality \cite{BRS}, the zero set $X'\cap X$ of a generic
cross-section $X'$ (that is the image of $X$ under a fiber preserving
self-homeomorphism of the total space) is a $q\lambda$-comanifold in $X$.
We denote its class in the set $\Emb^{q\lambda}(X)$ by $\E(\xi)$ and its
image in the group $\Imm^{q\lambda}(X)$ by $E(\xi)$.
The latter can also be defined as $i^*i_!([\id_X])$, where $i\:X\emb\xi$ is
the inclusion of the zero cross-section into the total space, and
$[\id_X]\in\Imm^0(X)$ is the fundamental class.
The Hurewicz homomorphism (cf.\ \cite{M3}) obviously sends $E(\xi)$ to
the usual (twisted) Euler class $e(\xi)\in H^q(X;\Z_{\lambda^{\otimes q}})$,
where $\Z_\lambda$ denotes the integral local coefficient system associated
with $\lambda$.

\begin{lemma}\label{4.2} Let $\bar X\to X$ be a double covering between compact
polyhedra, and let $\lambda$ be the associated line bundle.
The Pontrjagin--Thom correspondence \ref{4.0}(a) sends

$\E(q\lambda)$ to the equivariant homotopy class of the composition
$\bar X\xr{\phi_{\bar X}}S^\infty_\circ\xr{\rho^\infty_q}S^{qT}$;

$E(q\lambda)$ to
$[\rho^\infty_q\phi_{\bar X}]\in\omega^{qT}_{\Z/2}(\bar X_+)$.
\end{lemma}

\begin{proof}
$E(q\lambda)=E(\lambda)^q=\phi_X^*E(\gamma)^q=\phi_X^*E(q\gamma)$, where
$\phi_X\:X\to\RP^\infty$ classifies $\lambda$ and $\gamma$ is the universal
line bundle.
By similar considerations, $\E(q\lambda)=\phi_X^*\E(q\gamma)$ as well.
By a direct geometric construction, $\E(q\gamma)$ is represented by
$\RP^{\infty-q}$ with the canonical $q\lambda$-framing.
On the other hand, $(\rho^\infty_q)^{-1}(0)=S^{\infty-q}_\circ$. 
\end{proof}

\begin{remark} By Lemma \ref{4.2} the Pontrjagin--Thom isomorphism \ref{4.0}(a) sends
the obstruction $\Theta(f)\in\omega^{kT}_{\Z/2}({\Delta_f}_+)$ from
Theorem \ref{3.3} to $E(k\lambda)\in\Imm^{k\lambda}(\Delta_f/t)$, where $\lambda$
is the line bundle associated with the double covering
$\Delta_f\to\Delta_f/t$.
On the other hand, the Poincar\'e duality \ref{4.0}(b) obviously sends
$E(k\lambda)$ to the obstruction
$O(f)\in\Omega^{\sfr(n-d)}_d(\Delta_f/t;\lambda)$, defined in \cite{A2} in
the case of a map $f\:S^n\to\R^{2n-d}$.
\end{remark} 

Let $\phi\:Q\imm X$ be an immersed $q\lambda$-comanifold representing
some $\Phi\in\Imm^{q\lambda}(X)$.
By PL transversality, the double point immersion $d\phi\:\Delta_\phi\imm Q$
is an immersed $q\lambda$-comanifold in $Q$, hence the composition
$\Delta_\phi\xr{d\phi}Q\xr{\phi}X$ is an immersed $2q\lambda$-comanifold
in $X$.
Let $\Delta_\Phi$ denote its class in $\Imm^{2q\lambda}(X)$.

\begin{lemma}[Herbert's formula] \label{4.3}
$\Delta_\Phi+\phi_!E(\nu_\phi)=\Phi^2$.
\end{lemma}

Compare \cite{LS}, \cite{He}, \cite{EG}.
We will be slightly sloppy about notation in this proof, since spelling out
all the conventions would only make it less readable it seems.

\begin{proof} By the geometric definition of the cup product,
$\Phi\Cup\Phi=\phi_!\phi^*(\Phi)$.
We can think of $\Phi$ as $\phi_!([\id_Q])$.
Now $\phi$ factors into the composition $Q\xr{i}\nu_\phi\xr{\bar\phi}X$, so
$\phi^*\phi_!=i^*\bar\phi^*\bar\phi_!i_!$, where the domain of $\phi_!$ and
the range of $i_!$ is $\Imm^{q\lambda}(\nu_\phi,\partial\nu_\phi)$.
Since $\bar\phi$ is a codimension zero immersion,
$\bar\phi^*\bar\phi_!=1+(d\bar\phi)_!T(d\bar\phi)^*$, where
$$T=t^*=t_!\:
\Imm^{q\lambda}(\Delta_{\bar\phi},\Delta_{\bar\phi}\cap\partial\nu_\phi)\to
\Imm^{q\lambda}(\Delta_{\bar\phi},
\Cl{\partial\Delta_{\bar\phi}\but\partial\nu_\phi})$$
is induced by the involution $t\:\Delta_{\bar\phi}\to\Delta_{\bar\phi}$.
Now $i^*i_!([\id_Q])$ is by definition $E(\nu_\phi)$, so
$$\phi^*\phi_!([\id_Q])=
E(\nu_\phi)+i^*(d\bar\phi)_!T(d\bar\phi)^*i_!([\id_Q]).$$

We have $\Delta_{\bar\phi}\simeq\nu_{d\phi}\oplus\nu_{d\phi}$ as PL disc
bundles over $\Delta_\phi$.
Let $j\:\nu_{d\phi}\emb\Delta_{\bar\phi}$ be the inclusion onto one of
the factors, and let $\widebar{d\phi}\:\nu_{d\phi}\imm Q$ be the
extension of $d\phi$.
Then $(d\bar\phi)^*i_!=j_!\widebar{d\phi}^*$ and
$i^*(d\bar\phi)_!=\widebar{d\phi}_!j^*$.
Hence $i^*(d\bar\phi)_!T(d\bar\phi)^*i_!=
\widebar{d\phi}_!j^*Tj_!\widebar{d\phi}$.
Finally, $\phi\widebar{d\phi}=\widebar{d\bar\phi}j$, so we obtain
$$\phi_!\widebar{d\phi}_!j^*Tj_!\widebar{d\phi}([\id_Q])=
\widebar{d\bar\phi}_!j_!j^*([tj])=\widebar{d\bar\phi}_!([j]\Cup[tj])=
\Delta_\Phi,$$
where
$$\smallsmile\:
\Imm^{q\lambda}(\Delta_{\bar\phi},\Delta_{\bar\phi}\cap\partial\nu_\phi)
\otimes
\Imm^{q\lambda}(\Delta_{\bar\phi},
\Cl{\partial\Delta_{\bar\phi}\but\partial\nu_\phi})
\to\Imm^{2q\lambda}(\Delta_{\bar\phi},\partial\Delta_{\bar\phi}).$$
\end{proof}

\subsection{Secondary obstruction}
Let again $\lambda$ be the line bundle associated to a double covering
$\bar X\to X$.
Assuming that $E(q\lambda)=0$, we will construct a secondary obstruction to
the vanishing of $\E(q\lambda)$.
Let $Q\emb X$ be an embedded $q\lambda$-comanifold representing
$\E(q\lambda)\in\Emb^{q\lambda}(X)$.
Let $\phi\:R\imm X\x I$ be a generic immersed $q\lambda$-null-cobordism of
$Q$ given by the hypothesis $E(q\lambda)=0$.
The immersed $2q\lambda$-manifold $\Delta_\phi\imm R\imm X\x I$
represents an element of $\Imm^{2q\lambda}(X\x I,X\x\partial I)$.
The Thom isomorphism sends it to an element
$$F(q\lambda)\in\Imm^{2q\lambda-1}(X)\simeq\omega^{2qT-1}_{\Z/2}(\bar X_+).$$

\begin{proposition}\label{4.4} $F(q\lambda)$ is well-defined if $E(q\lambda)=0$,
and vanishes if $\E(q\lambda)=0$.
\end{proposition}

\begin{proof} The second assertion holds by construction.
To prove the first, let us temporarily denote $F(q\lambda)$ by $F(R)$, and
let us consider another immersed $q\lambda$-null-cobordism $\phi'\:R'\imm X\x I$
of $Q$.
Let $\tau\:I=[0,1]\to [-1,0]$ be defined by $x\mapsto -x$, and let
$\phi''=(\id_X\x\tau)\phi'\:R'\imm X\x[-1,0]$.
Let $\psi\:W\imm X\x [-1,1]$ be the immersed $q\lambda$-comanifold
$\phi\cup(-\phi'')$.
Then $F(W)=F(R)-F(R')$.
We may assume that $W$ is contained in $X\x [-1,1-\eps]$ for some $\eps>0$.
Since $W$ is regularly homotopic into $X\x [1-\eps,1]$ by an ambient isotopy,
$[\psi]^2=0$ and so $F(W)+\psi_!E(\nu_\psi)=0$ by the Herbert formula \ref{4.3}.
Now $W$ is $q\lambda$-framed, so $\nu_\psi$ is isomorphic to
$q(\psi^*\lambda)=\psi^*(q\lambda)$.
Hence $E(\nu_\psi)=\psi^*E(q\lambda)=0$ by the hypothesis.
Thus $F(W)=0$. 
\end{proof}

\begin{example}\label{4.5} Let $M$ be a closed $(2k-1)$-manifold and $\lambda$
a line bundle over $M$ with $E(k\lambda)=0$.
Then $F(k\lambda)=0$.
\end{example}

A stronger result will be obtained in Theorem \ref{4.7} by a different method.

\begin{proof} $F(k\lambda)$ lies in $\Imm^{2k\lambda-1}(M)\simeq H^{2k-1}(M)$.
On the other hand, if $F(k\lambda)$ is represented by $\Delta_\phi\imm M\x I$
for some generic immersed $k\lambda$-null-cobordism $\phi\:R\imm M\x I$ of
a representative of $\E(k\lambda)$, the double covering
$p\:\Delta_\phi\to\Delta_\phi/t$ is trivial, so
$F(k\lambda)=2[\Delta_\phi/t]$ if $t$ preserves the orientation of the bundle
$p^*(2k\lambda)$, and $0$ otherwise.
If $k$ is odd, the orientation is reversed.
If $M$ is non-orientable, $H^{2k-1}(M)\simeq\Z/2$ and again $F(k\lambda)=0$.

If $k$ is even and $M$ is orientable, $R$ is also orientable, and
$F(k\lambda)$ is twice the algebraic number of double points of $\phi$.
Let $\bar M$ be the double cover of $M$ corresponding to $\lambda$ and
$\bar R\imm\bar M\x I$ the double cover over $\phi$.
Since the coverings $\bar R\to R$ and $\bar M\to M$ preserve orientations,
$2F(k\lambda)$ is twice the algebraic number of double points of $\bar R$.
Since $\bar M\to S^{\infty}$ is non-equivariantly null-homotopic,
by the non-equivariant version of Lemma \ref{4.2}, $\partial\bar R$ admits
an embedded framed null-cobordism $\bar R'$.
By the non-equivariant version of the proof of Proposition \ref{4.4} (which is
easier since all normal bundles are trivial), $\bar R$ has as many double
points as $\bar R'$, that is algebraically zero.
So again $F(k\lambda)=0$. 
\end{proof}

Using Lemma \ref{4.4}, we will now prove that $F(k\lambda)$ vanishes identically
whenever it is defined.

\begin{lemma}\label{4.6} Let $\lambda$ be a line bundle over a polyhedron $X$
such that $E(k\lambda)=0$.
There exists a singular $0$-comanifold $f\:X'\to X$ such that
$[f]=[\id_X]\in\Imm^0(X)$ and the double cover $\bar X'$ corresponding to
$f^*\lambda$ admits an equivariant map to $S^{k-1}_\circ$.
\end{lemma}

\begin{remark} In the case where $\bar X=\Delta_f$ for some $k$-realizable
stable $f\:N\to\R^m$, we can take $\bar X'=\Delta_{\pi g}$, where
$g\:N\emb\R^{m+k}$ is an embedding such that the composition
$N\xr{g}\R^{m+k}\xr{\pi}\R^m$ is stable and $C^0$-close to $f$.
A bordism between $\id\:\bar X\to\bar X$ and the projection $\bar X'\to\bar X$
is given by the projection of $\Delta_H$, where $H$ is a generic homotopy
between $f$ and $\pi g$.
\end{remark}

\begin{proof} By Lemma \ref{4.2}, the composition
$\bar X\xr{\phi_{\bar X}}S^\infty_\circ\xr{\rho^\infty_k}S^{kT}$ is
stably equivariantly null-homotopic.
That is, for some $m$, there exists an equivariant homotopy
$H_t\:S^{mT}\wedge\bar X_+\to S^{(m+k)T}$ between the suspension of
$\rho^\infty_k\phi_{\bar X}$ and the constant map to the basepoint $b$.
By using the composition
$$S^{m-1}_\circ*\bar X\xr{r_1}
(S^{m-1}_\circ*\bar X)\cup_{S^{m-1}_\circ}(S^{m-1}_\circ\x I)
\xr{r_2}(S^{mT}\wedge\bar X_+)\vee (S^{m-1}_\circ*b),$$
where $r_1$ is the obvious surjection and $r_2$ shrinks $S^{m-1}_\circ$
to the basepoint, it follows that the composition
$S^{m-1}_\circ*\bar X\xr{\phi_{S^{m-1}*\bar X}}S^{m+\infty}_\circ
\xr{\rho^{m+\infty}_{m+k}}S^{(m+k)T}$ is equivariantly homotopic to the join
$j\:S^{m-1}_\circ*\bar X\to S^{m-1}_\circ*b$ of the identity and the constant
map to $b$ by a homotopy $h_t$ which is {\it suspension sphere preserving},
i.e.\ for each $t$ sends the suspension sphere $S^{m-1}_\circ$ identically
onto $S^{m-1}_\circ\subset S^{(m+k)T}\but S^{kT}$.
Since its final map $h_1$ is null-homotopic, by Lemma \ref{3.2}(a) with
$L=\emptyset$, if $m$ is large enough, there exists an equivariant map
$S^{m-1}_\circ*\bar X\to S^{m+k-1}_\circ\subset S^{(m+k)T}$.
By general position we may assume that this map is suspension sphere
preserving, and is homotopic to $j$ by a suspension sphere preserving
homotopy $h'_t$.
Combining $h_t$ and $h'_t$, we get an equivariant suspension sphere
preserving homotopy
$H\:(S^{m-1}_\circ*\bar X)\x I\to S^{(m+k)T}$ from
$\rho^{m+\infty}_{m+k}\phi_{S^{m-1}_\circ*\bar X}$ to a map into
$S^{k+m-1}_\circ\subset S^{(k+m)T}$.
Assuming that this homotopy is generic, let $\bar W:=H^{-1}(S^{kT})$.
Since $\rho^{m+\infty}_{m+k}$ restricts to $\rho^\infty_k$ over $S^{kT}$, and
$\phi_{S^{m-1}_\circ*\bar X}$ may be taken to be the suspension of
$\phi_{\bar X}$, $\bar W$ meets $(S^{m-1}_\circ*\bar X)\x\{0\}$ in $\bar X$.
By construction, the intersection $\bar X'$ of $\bar W$ with
$(S^{m-1}_\circ*\bar X)\x\{1\}$, admits an equivariant map to
$S^{k-1}_\circ$.
Finally, $\bar W$ lies in $(mT\x\bar X)\x I\subset (S^{m-1}_\circ*\bar X)\x I$, so
its projection onto $\bar X\x I$ is an equivariant singular $0$-cobordism
between $[\bar X]$ and $[H_1\:\bar X'\to\bar X]$. 
\end{proof}

\begin{theorem}\label{4.7} Let $\lambda$ be a line bundle over a polyhedron $X$
such that $E(k\lambda)=0$.
Then $F(k\lambda)=0$.
\end{theorem}

\begin{proof} Let $f$ be given by Lemma \ref{4.6}.
We have $\E(kf^*\lambda)=0$, so $F(kf^*\lambda)=0$.
On the other hand, if $w\:W\to X\x I$ is a singular $0$-cobordism between
$[\id_X]$ and $[f]$, then $w_!F(kw^*\lambda)$ is represented by a
$k\lambda$-cobordism between some representative of $F(k\lambda)$ and
$f\phi$, where $\phi$ is some representative of $F(kf^*\lambda)$.
Since $F(kf^*\lambda)$ is well-defined, $\phi$ is
$kf^*\lambda$-null-cobordant, hence $f\phi$ is $k\lambda$-null-cobordant.
\end{proof}

\begin{remark} By Lemma \ref{4.2} and Lemma \ref{3.2}(a), $E(k\lambda)=0$ is equivalent
to existence of an equivariant map
$g\:S^{\infty-1}_\circ*\bar X\to S^{\infty+k-1}_\circ$.
We do not know if the suspension of the equivariant map
$\bar X'\to S^{k-1}_\circ$ given by Lemma \ref{4.6} can be chosen homotopic
to the composition $S^{\infty-1}_\circ*\bar X'\xr{\Sigma f}
S^{\infty-1}_\circ*\bar X\xr{g}S^{\infty+k-1}_\circ$.
If this were the case, the proof given by Theorem \ref{4.7} that $\E(k\lambda)$
bounds an immersed $k\lambda$-null-cobordism $\phi\:R\imm X\x I$ with
$[\Delta_\phi]=0$ would not require the use of Proposition \ref{4.4}.
The same argument would then work for
$[\Delta_\phi/t]\in\Imm^{k\lambda\otimes(\bigstar+1)-1}(X)$ in place of
$[\Delta_\phi]$, where $\bigstar$ stands for a line bundle, which is a part
of the data of the immersed comanifold $\Delta_\phi/t$ (namely, it is
associated to the double covering $\Delta_\phi\to\Delta_\phi/t$).
\end{remark}

\subsection{Proof of Theorem \ref{th6}}

\begin{proposition}\label{4.8} Let $M$ be a closed non-orientable
$(2k-1)$-manifold and $\lambda$ a line bundle over $M$.
If $k$ is even, $\Imm^{k\lambda}(M)=\Emb^{k\lambda}(M)$.
\end{proposition}

\begin{proof} Let $O$ be a connected codimension one submanifold in $\R^{k-1}$
(e.g.\ a point when $k=2$ and $S^{k-2}$ when $k>2$).
Let $\ell$ denote $S^1$, then $\ell\x O\subset\ell\x\R^{k-1}\subset\ell\x\R^{2k-2}$
is $k$-framed in $\ell\x\R^{2k-2}$.
This extends to a $k$-framing of $\mu\x O$ in $\mu\x\R^{2k-2}$, where $\mu$
is the mapping cylinder of the nontrivial double covering $\bar\ell\to\ell$.
Since $\bar\ell\x O\to\ell\x O$ is a double covering, if we twist
the (integer) framing of $\ell\x O$ in $\ell\x\R^{2k-2}$ by one full twist,
the (integer) framing of $\bar\ell\x O$ in $\bar\ell\x\R^{2k-2}$ will differ
from the original one by two full twists.
These two framed embeddings $\bar\ell\x O\emb\bar\ell\x\R^{2k-2}$ can be
joined by a framed regular homotopy with one transverse double point in
$\bar\ell\x\R^{2k-2}\x I$.
Combining the two framed embeddings $\mu\x O\emb\mu\x\R^{2k-2}$ with this
regular homotopy, we obtain a framed immersion $K\x O\imm K\x\R^{2k-2}$
with one double point, where $K=\mu\cup\bar\ell\x I\cup\mu$ is the Klein
bottle.
Since $K$ is the boundary of the nontrivial $2$-disk bundle over $\ell$,
$K\x\R^{2k-2}$ embeds into the nontrivial $(2k-1)$-vector bundle over
$\ell$.
Thus if $\ell$ is an orientation reversing loop in $M\x I$, we obtain
a $k$-framed immersion with one double point
$K\x O\imm K\x\R^{2k-2}\subset\nu_\ell$ in the regular neighborhood of $\ell$.
Since $k$ is even, $k\lambda$ restricts to the trivial bundle over $\ell$,
so this immersion is also $k\lambda$-framed in $M\x I$.

By general position, every element of $\Imm^{k\lambda}(M)$ can be
represented by an embedded $k\lambda$-framed manifold $Q$ in $M$.
Suppose that it admits an immersed $k\lambda$-null-cobordism $R\imm M\x I$.
If $R$ has an odd number of double points, we replace it by $R\cup K\x O$.
Since $M$ is non-orientable, the double points of $R$ carry no global signs
and so can now be paired up so as to match the local signs.
More precisely, each double point lifts to a pair of double points of
opposite signs of the immersion $\bar R\imm\bar M\x I$ in the orientation
double cover.
Since the number of pairs is even, we can pick one double point from each
pair so that the total algebraic number of the picked points is zero.
Then the picked points can be paired up with signs, and cancelled along
framed $1$-handles in $\bar M$.
These project to $k\lambda$-framed handles in $M$, killing all double points
of $R$. 
\end{proof}

\begin{theorem}\label{4.9} Let $X$ be a $(2k-1)$-polyhedron and $\lambda$ a
line bundle over $X$ with $E(k\lambda)=0$.
If $k$ is even, $\E(k\lambda)=0$.
\end{theorem}

\begin{proof} By Theorem \ref{4.7}, $F(k\lambda)=0$.
Let $\phi\:R\imm X\x I$ be an immersed $k\lambda$-null-cobordism of
a representative of $\E(k\lambda)$.
Since $k$ is even, $F(k\lambda)=2[\Delta_\phi/t]\in H^{2k-1}(X)$ (cf.\ the
proof of Example \ref{4.5}).

First assume that $H^{2k-1}(X)$ contains no elements of order $2$.
Then the algebraic number of double points of $R$ is zero.
So they can be paired up with signs and cancelled by a surgery along
$k\lambda$-framed $1$-handles (cf.\ the proof of Proposition \ref{4.8}; if
$H^{2k-1}(X)=0$ say, the double points can be pushed off the boundary).
Thus $Q=\partial R$ bounds an embedded $k\lambda$-null-cobordism.

Now in the general case it suffices to prove that for each generator
$[D]\in H^{2k-1}(X)$ of order $2n$ there is
an immersed $k\lambda$-comanifold in $X\x I$ with $n$ double points each
representing $\delta^*[D]\in H^{2k}(X\x I,X\x\partial I)$.
Since $2n[D]=0$, there is a $(2k-1)$-framed comanifold $C\imm X\x I$ with
boundary the constant map $\{1,\dots,2n\}\to D\subset X\x\{0\}$.
Let us pick a free involution $t$ on the $2n$ points $\partial C$.
It extends to a free involution on a double cover $\bar\ell$ of $\ell:=C/t$.
An embedded perturbation $\ell\subset X\x I$ of $C/t$ is clearly
$(2k-2+\mu)$-framed in $X$, where $\mu$ is the line bundle associated with
the double covering $\bar\ell\to\ell$.
The remainder of the construction repeats that in the proof of
Proposition \ref{4.8}.
\end{proof}

\begin{corollary*}[=Theorem \ref{th6}] Let $N^n$ be a compact smooth manifold (resp.\ a compact
polyhedron), $M^{2n-2k+1}$ a smooth (PL) manifold and $f\:N\to M$ a stable smooth (PL) map, 
where $n\ge 2k+1$.
In the smooth case, assume additionally that either $f$ is a fold map or $n\ge 3k-2$.
If $f\:N\to M$ is $k$-realizable and $k\in\{2,4,8\}$, then $f$ is a smooth (PL) $k$-prem.
\end{corollary*}

\begin{proof} Let $\lambda$ be the line bundle associated with the $2$-cover
$\Delta_f\to\Delta_f/t$ and set $X=\Delta_f/t$.
By the trivial direction of Theorem \ref{3.3}(a), $\Theta(f)=0$, hence by
Lemma \ref{4.2}, $E(k\lambda)=0$.
Then by Theorem \ref{4.9}, $\E(k\lambda)=0$.
Therefore by Lemma \ref{4.2}, the composition
$X\xr{\phi_X}S^\infty_\circ\xr{\rho^\infty_k}S^{kT}$ is equivariantly
null-homotopic.
By Lemma \ref{3.7}, $X$ admits an equivariant map to $S^{k-1}_\circ$.
By Theorem \ref{2.1}, $f$ is a $k$-prem. 
\end{proof}


\begin{thebibliography}{00}

\bib{Ad}{book}{
   author={Adachi, M.},
   title={Embeddings and Immersions},
   publisher={Iwanami Shoten},
   place={Tokyo},
   date={1984},
   language={Japanese},
   translation={
      series={Transl. Math. Monogr.},
      volume={124},
      publisher={Amer. Math. Soc.},
      date={1993},
      }
}

\bib{Akh}{article}{
   author={Akhmet'ev, P. M.},
   title={Prem-mappings, triple self-interesection points of an oriented surface, 
   and Rokhlin's signature theorem},
   journal={Mat. Zametki},
   volume={59},
   date={1996},
   number={6},
   pages={803--810, 959, \href{http://mi.mathnet.ru/mz1779}{MathNet}},
   translation={
      language={English},
      journal={Math. Notes},
      volume={59},
      date={1996},
      pages={581--585},
   },
}

\bib{A1}{article}{
   author={Akhmet'ev, P. M.},
   title={On isotopic and discrete realization of maps of $n$-dimensional sphere in Euclidean space},
   journal={Mat. Sb.},
   volume={187},
   date={1996},
   number={7},
   pages={3--34; \href{http://mi.mathnet.ru/msb142}{MathNet}},
   translation={
      journal={Sb. Math.},
      volume={187},
      date={1996},
      pages={951--980},
   },
}

\bib{A2}{article}{
   author={Akhmetiev, P. M.},
   title={Pontryagin--Thom construction for approximation of mappings by embeddings},
   journal={Topol. Appl.},
   volume={140},
   date={2004},
   pages={133--149},
   note={\href{https://www.sciencedirect.com/science/article/pii/S0166864103002876}{Journal}}
}

\bib{AE}{article}{
   author={Akhmet'ev, P. M.},
   author={Eccles, P. J.},
   title={The relationship between framed bordism and skew-framed bordism},
   journal={Bull. Lond. Math. Soc.},
   volume={39},
   date={2007},
   pages={473--481},
}

\bib{AMR}{article}{
   author={Akhmetiev, P. M.},
   author={Melikhov, S. A.},
   author={Repov{\v{s}}, D.},
   title={Projected embeddings and near-projected embeddings},
   note={Preprint (2007)},
}

\bib{ARS}{article}{
   author={Akhmetiev, P. M.},
   author={Repov{\v{s}}, D.},
   author={Skopenkov, A. B.},
   title={Obstructions to approximating maps of $n$-manifolds into $\R^{2n}$ by embeddings},
   journal={Topol. Appl.},
   volume={123},
   date={2002},
   pages={3--14},
   note={\href{http://www.sciencedirect.com/science/article/pii/S016686410100164X}{Journal}},
}

\bib{BRS}{book}{
   author={Buoncristiano, S.},
   author={Rourke, C. P.},
   author={Sanderson, B. J.},
   title={A Geometric Approach to Homology Theory},
   series={London Math. Soc. Lecture Note Series},
   volume={18},
   publisher={Cambridge Univ. Press},
   date={1976},
}

\bib{CS}{article}{
   author={Carter, J. S.},
   author={Saito, M.},
   title={Surfaces in $3$-space that do not lift to embeddings in $4$-space},
   conference={
      title={Knot theory},
      address={Warsaw},
      date={1995},
   },
   book={
      series={Banach Center Publ.},
      volume={42},
      publisher={Polish Acad. Sci. Inst. Math.},
      place={Warsaw},
   },
   date={1998},
   pages={29--47},
}

\bib{CF}{article}{
   author={Conner, P. E.},
   author={Floyd, E. E.},
   title={Fixed point free involutions and equivariant maps},
   journal={Bull. Amer. Math. Soc.},
   volume={66},
   date={1960},
   pages={416--441},
   note={\href{https://www.ams.org/journals/bull/1960-66-06/S0002-9904-1960-10492-2/}{Journal}}
}

\bib{EG}{article}{
   author={Eccles, P. J.},
   author={Grant, M.},
   title={Bordism groups of immersions and classes represented by self-intersections},
   journal={Algebr. Geom. Topol.},
   volume={7},
   date={2007},
   pages={1081--1097},
   note={\arxiv[AT]{0504152}},
}

\bib{FFG}{book}{
   author={D. B. Fuchs},
   author={A. T. Fomenko},
   author={V. L. Gutenmacher},
   title={Homotopic Topology},
   publisher={Moscow State Univ.},
   date={1969},
   language={Russian},
   translation={
     publisher={Akad\'emiai Kiad\'o},
     address={Budapest},
     date={1986},
     }
   note={(Revised editions: Nauka, Moscow, 1989; Grad. Texts in Math., vol. 273, Springer, 2016)}  
}

\bib{GWPL}{book}{
   author={Gibson, C. G.},
   author={Wirthm\"{u}ller, K.},
   author={du Plessis, A. A.},
   author={Looijenga, E. J. N.},
   title={Topological Stability of Smooth Mappings},
   series={Lecture Notes in Math.},
   volume={552},
   publisher={Springer-Verlag},
   address={Berlin},
   date={1976},
}

\bib{Gi}{article}{
   author={Giller, C. A.},
   title={Towards a classical knot theory for surfaces in $\R^4$},
   journal={Illinois J. Math.},
   volume={26},
   date={1982},
   pages={591--631},
}

\bib{GG}{book}{
   author={Golubitsky, M.},
   author={Guillemin, V.},
   title={Stable Mappings and their Singularities},
   series={Grad. Texts in Math.},
   volume={14},
   publisher={Springer},
   date={1973},
   translation={
    language={Russian},
    publisher={Mir},
    date={1977},
    }
}

\bib{Hae}{article}{
   author={Haefliger, A.},
   title={Plongements diff\'{e}rentiables de vari\'{e}t\'{e}s dans vari\'{e}t\'{e}s},
   journal={Comment. Math. Helv.},
   volume={36},
   date={1961},
   pages={47--82},
   note={\href{https://eudml.org/doc/139228}{EuDML}},
}

\bib{Han}{article}{
   author={Hansen, V. L.},
   title={Embedding finite covering spaces into trivial bundles},
   journal={Math. Ann.},
   volume={236},
   date={1978},
   pages={239--243},
   note={\href{http://www.digizeitschriften.de/dms/resolveppn/?PID=GDZPPN002316218}{DigiZeitschriften}},
}

\bib{He}{article}{
   author={Herbert, R. J.},
   title={Multiple points of immersed manifolds},
   journal={Mem. Amer. Math. Soc.},
   volume={34},
   date={1981},
   number={250},
}
\bib{Hi}{article}{
   author={Hirsch, M. W.},
   title={On piecewise linear immersions},
   journal={Proc. Amer. Math. Soc.},
   volume={16},
   date={1965},
   pages={1029--1030},
   note={\href{https://www.ams.org/journals/proc/1965-016-05/S0002-9939-1965-0189050-X/}{Journal}}
}

\bib{LS}{article}{
   author={Lashof, R. K.},
   author={Smale, S.},
   title={Self-intersections of immersed manifolds},
   journal={J. Math. Mech.},
   volume={8},
   date={1959},
   pages={143--157},
   note={\href{http://www.iumj.indiana.edu/IUMJ/fulltext.php?artid=58010&year=1959&volume=8}{Journal}}
}

\bib{Ko}{article}{
   author={Koschorke, U.},
   title={On link maps and their homotopy classification},
   journal={Math. Ann.},
   volume={286},
   date={1990},
   pages={753--782},
   note={\href{http://www.digizeitschriften.de/dms/resolveppn/?PID=GDZPPN002333821}{DigiZeitschriften}}
}

\bib{KS}{article}{
   author={Koschorke, U.},
   author={Sanderson, B.},
   title={Geometric interpretations of the generalized Hopf invariant},
   journal={Math. Scand.},
   volume={41},
   date={1977},
   pages={199--217},
   note={\href{https://www.mscand.dk/article/view/11714}{Journal}}
}

\bib{Kul}{article}{
   author={Kulinich, E. V.},
   title={On topologically equivalent Morse functions on surfaces},
   journal={Methods Funct. Anal. Topology},
   volume={4},
   date={1998},
   number={1},
   pages={59--64},
   note={\href{http://mfat.imath.kiev.ua/article/?id=34}{Journal}}
}

\bib{Ma}{article}{
   author={Mather, J. N.},
   title={Stability of $C^\infty$ mappings. VI: The nice dimensions},
   conference={
      title={Proc. of Liverpool Singularities Symp. I},
   },
   book={
      series={Lecture Notes in Math.},
      volume={192},
      publisher={Springer},
      place={Berlin},
   },
   date={1971},
   pages={207--253}
}

\bib{M+}{book}{
   author={May, J. P.},
   author={et al.}
   title={Equivariant Homotopy and Cohomology Theory},
   series={CBMS Reg. Conf. Ser. in Math.},
   volume={91},
   publisher={Amer. Math. Soc.}
   place={Providence, RI},
   date={1996},
}

\bib{M1}{article}{
   author={Melikhov, S. A.},
   title={On maps with unstable singularities},
   journal={Topol. Appl.},
   volume={120},
   date={2002},
   pages={105--156},
   note={\arxiv[GT]{0101047}}
}

\bib{M2}{article}{
   author={Melikhov, S. A.},
   title={Sphere eversions and realization of mappings},
   journal={Proc. Steklov Inst. Math.},
   date={2004},
   number={247},
   pages={143--163; \arxiv[GT]{0305158}},
   translation={
      language={Russian},
      journal={Tr. Mat. Inst. Steklova},
      volume={247},
      date={2004},
      pages={159--181},
      note={\href{http://mi.mathnet.ru/tm15}{MathNet}},
   },
}

\bib{M3}{article}{
   author={Melikhov, S. A.},
   title={The van Kampen obstruction and its relatives},
   journal={Tr. Mat. Inst. Steklova},
   volume={266},
   date={2009},
   pages={149--183},
   note={Reprinted in: Proc. Steklov Inst. Math. {\bf 266} (2009), 142--176; \arxiv[GT]{0612082}},
}

\bib{M4}{article}{
   author={Melikhov, S. A.},
   title={Transverse fundamental group and projected embeddings},
   journal={Proc. Steklov Inst. Math.},
   date={2015},
   number={290},
   pages={155--165; \arxiv{1505.00505}},
   translation={
      language={Russian},
      journal={Tr. Mat. Inst. Steklova},
      volume={290},
      date={2015},
      pages={166--177},
      note={\href{http://mi.mathnet.ru/tm3646}{MathNet}}
   },
}

\bib{M5}{article}{
   author={Melikhov, S. A.},
   title={Lifting generic maps to embeddings: I. Triangulation and smoothing;
   II. The double point obstruction},
   note={\arxiv{2011.01402v1} and \arxiv{1711.03518v4}}
}

\bib{Po1}{article}{
   author={Po\'{e}naru, V.},
   title={Some invariants of generic immersions and their geometric applications},
   journal={Bull. Amer. Math. Soc.},
   volume={81},
   date={1975},
   pages={1079--1082},
   note={\href{https://www.ams.org/journals/bull/1975-81-06/S0002-9904-1975-13923-1/}{Journal}}
}

\bib{Po2}{book}{
   author={Po\'{e}naru, V.},
   title={Homotopie re\'guli\`ere et isotopie},
   series={Publ. Math. Orsay},
   date={1974},
   publisher={Univ. Paris XI, U.E.R. Math.},
   place={Orsay},
   note={No. 106 (23), 54 pp.; \href{http://sites.mathdoc.fr/PMO/feuilleter.php?id=PMO_1974}{Mathdoc}}
}

\bib{Po}{book}{
   author={Pontryagin, L. S.},
   title={Smooth Manifolds and their Applications in Homotopy Theory},
   series={Trudy MIAN},
   volume={45},
   date={1955; \href{http://mi.mathnet.ru/tm1181}{MathNet}},
   translation={
     series={Amer. Math. Soc. Transl., Ser. 2},
     volume={11},
     date={1959},
     }
}

\bib{RS}{article}{
   author={Repov\v{s}, D.},
   author={Skopenkov, A. B.},
   title={A deleted product criterion for approximability of maps by embeddings},
   journal={Topol. Appl.},
   volume={87},
   date={1998},
   pages={1--19},
   note={\href{https://www.sciencedirect.com/science/article/pii/S0166864197001211}{Journal}}
}

\bib{RS1}{article}{
   author={Rourke, C. P.},
   author={Sanderson, B. J.},
   title={Block bundles. II. Transversality},
   journal={Ann. of Math.},
   volume={87},
   date={1968},
   pages={256--278},
}

\bib{RS2}{article}{
   author={Rourke, C.},
   author={Sanderson, B.},
   title={The compression theorem. I},
   journal={Geom. Topol.},
   volume={5},
   date={2001},
   pages={399--429},
   note={\arxiv[GT]{9712235}}
}

\bib{S}{article}{
   author={Satoh, S.},
   title={Lifting a generic surface in 3-space to an embedded surface in 4-space},
   journal={Topology Appl.},
   volume={106},
   date={2000},
   pages={103--113},
   note={\href{https://www.sciencedirect.com/science/article/pii/S0166864199000747}{Journal}}
}

\bib{Si}{article}{
   author={Siek\l ucki, K.},
   title={Realization of mappings},
   journal={Fund. Math.},
   volume={65},
   date={1969},
   pages={325--343},
   note={\href{https://eudml.org/doc/214113}{EuDML}}
}

\bib{Sz}{article}{
   author={Sz\H{u}cs, A.},
   title={Multiple points of singular maps},
   journal={Math. Proc. Cambridge Philos. Soc.},
   volume={100},
   date={1986},
   pages={331--346},
}

\bib{Sz2}{article}{
   author={Sz\H{u}cs, A.}
   title={On the cobordism groups of immersions and embeddings},
   journal={Math. Proc. Cambridge Philos. Soc.},
   volume={109},
   date={1991},
   pages={343--349},
}

\bib{TV}{article}{
 author = {Tarasov, S. P.},
 author={Vyalyi, M. N.},
 title = {Some PL functions on surfaces are not height functions},
 conference={
  title = {Proc. 13th Annual Symp. on Comput. Geometry},
  place = {Nice, France},
 },
 book={
  publisher = {ACM},
  address = {New York},
  date = {1997},
 },
 pages = {113--118},
} 

\bib{Ya}{article}{
   author={Yamamoto, M.},
   title={Lifting a generic map of a surface into the plane to an embedding
   into 4-space},
   journal={Illinois J. Math.},
   volume={51},
   date={2007},
   pages={705--721},
   note={\href{https://projecteuclid.org/euclid.ijm/1258131098}{ProjectEuclid}},
}

\bib{Ada}{article}{
   author={Adams, J. F.},
   title={Prerequisites (on equivariant stable homotopy) for Carlsson's lecture},
   conference={
      title={Algebraic Topology},
      address={Aarhus},
      date={1982},
   },
   book={
      series={Lecture Notes in Math.},
      volume={1051},
      publisher={Springer},
      place={Berlin},
   },
   date={1984},
   pages={483--532},
}


\end{thebibliography}
\end{document}